\documentclass{article}
\usepackage{amsmath}
\usepackage{amssymb}
\usepackage{authblk}
\usepackage{color}
\usepackage{xcolor}
\usepackage{graphicx}
\usepackage{latexsym}
\usepackage{enumerate}
\usepackage{layout}
\usepackage{amsthm} 
\usepackage{hyperref}
\usepackage[utf8]{inputenc}
\hypersetup{colorlinks=true, linkcolor=blue, urlcolor=blue, citecolor=blue}

\newcommand{\R}{\mathbb{R}}
\newcommand{\E}{\mathbb{E}}

\newtheorem{lemma}{Lemma}[section]

\newtheorem{proposition}{Proposition}[section]
\newtheorem{theorem}{Theorem}[section]
\newtheorem{remark}{Remark}

\newcommand\Normal{{\mathcal N}}
\newcommand{\Var}{{\rm Var}}
\newcommand{\Sphere}{\mathbb{S}}
\newcommand{\Pbb}{\mathbb{P}}

\usepackage{geometry}
\begin{document}

\title{On the Euler-Poincar\'e Characteristic of  the  planar Berry's random wave: fluctuations and a perturbation study}

\author[1]{Elena Di Bernardino\thanks{Elena.DI\_BERNARDINO@univ-cotedazur.fr}}
\author[1]{Radomyra Shevchenko\thanks{Radomyra.SHEVCHENKO@univ-cotedazur.fr}}
\author[2]{Anna Paola Todino\thanks{annapaola.todino@uniupo.it}} 
 
\affil[1]{Universit\'{e} C\^{o}te d'Azur, Laboratoire J.A. Dieudonn\'{e}, UMR CNRS  7351,  Nice, 06108, France}
\affil[2]{ Università del Piemonte Orientale,
Dipartimento di Scienze e Innovazione Tecnologica, Alessandria, 15121, Italy}  
  
\date{}

\maketitle 
\begin{abstract}
We prove the Central Limit Theorem for the Euler-Poincar\'e characteristic of Berry's random wave model in a growing domain. {We also show Gaussian fluctuations for a class of Berry's mixture models that correspond to a perturbation of the initial random field.} Finally, some statistical applications, explicit calculations of the variance of the perturbed Berry's model and numerical investigations are provided to  support our theoretical results.
\end{abstract}

\textbf{Keywords and Phrases:} Central Limit Theorem;
Euler-Poincar\'e  characteristic;
Excursion sets;
Random plane wave;
Wiener chaos expansion. \\

\textbf{AMS Classification:} 60G60, 60F05, 33C10, 60D05.

\maketitle
 
\section{Introduction and main results}
 
\subsection{Lipschitz-Killing Curvatures and Berry's random wave model} 
The study of random fields  in $\R^2$ through the geometry of their excursion sets
has received a lot of interest in recent literature. The three geometric  additive features $\mathcal L_k$, $k=0, 1, 2$   are referred to as either Lipschitz-Killing curvatures (LKCs) in
stochastic geometry, curvature measures in  differential geometry, intrinsic volumes
or Minkowski functionals in integral and convex geometry.  Theses  functionals  are extensively used in the literature to describe the geometry or the topology of the considered excursion set.  Loosely speaking,  $\mathcal L_0$ is the Euler-Poincaré characteristic of the excursion set (i.e., the difference
between the number of connected components and the number of holes),  $\mathcal L_{1}$ is half the  perimeter length and $\mathcal L_2$ is the area.\\
 
\textbf{The planar  Berry's random wave model.}  Let us denote by $T$ a bounded rectangle contained in $\mathbb{R}^2$, with non empty interior. Obviously,   
\begin{align}\label{eq:L012T}\mathcal L_{0}(T)=1, \quad \mathcal L_{1}(T)=\frac 12 |\partial T|_1, \quad \mathcal L_{2}(T)=|T|,
\end{align}
where $\partial T$ stands for the boundary of the set $T$, $|\cdot |$ the two-dimensional Lebesgue measure  and   $|\cdot |_1$ the one-dimensional Hausdorff measure.  \smallskip

We will consider here the random excursion set above a real level $u$, associated to the Berry's random wave model, i.e., 
\begin{equation}\label{EXCSET}
A(f, T, u)= \{x \in T: f(x)\geq u\},
\end{equation}
 
where $f$ is the isotropic centered Gaussian random field  $f: \Omega \times \mathbb{R}^2 \to \mathbb{R}$  uniquely defined by the covariance function
\begin{equation} \label{covariance}
\mathbb E [f(x)f(y)]=J_0(\|x-y\|),
\end{equation}
for $x,\, y\in\mathbb \R^2$, where $J_0$ is the order $0$ Bessel function. Notice that in our context, since $f$ is the Laplacian eigenfunction with eigenvalue $1$  (see \cite{NourdinPeccatiRossi2017}), it can be expressed through its second derivatives.
 
In this paper we focus on the asymptotic behaviour as the observation domain   $T$ tends to $\R^2$, of  the  first Lipschitz-Killing curvature (the Euler-Poincaré Characteristic, EPC) of the excursion set \eqref{EXCSET} for the Berry's random wave model in $\R^2$ defined in \eqref{covariance}. More precisely,  as we  are interested in the asymptotics as  $T \nearrow \R^2$,  instead of considering the EPC of the excursion of $f$ 
above $u$, we consider the \textit{modified Euler characteristic},   inspired by \cite{AT07} Lemma 11.7.1,  \cite{EstradeLeon16-AoP} Section 1, and \cite{dibernardino2017}.  Roughly speaking, by applying Morse's theorem, both notions coincide on  $\mathring T$, the interior of the domain $T$. Indeed, the Euler characteristic of $A(f, T, u)$ is equal to a sum of two terms (see \cite{AT07} Chapter 9, for instance).  The first one only depends on the restriction of $f$ to $\mathring T$, the second one exclusively depends on the behaviour of $f$ on the $l$-dimensional faces of $T$, with $0 \le l < 2$. From now on, we focus on the first term, denoted $\varphi(f,T,u)$,  and  defined in terms of number of local extrema of $f$ in $\mathring T$ above $u$,  as  
\begin{equation}\label{modifEPC}
\varphi(f,T,u)= \#\{\mbox{local extrema of $f$ above $u$ in $\mathring T$}\} - \#\{\mbox{local saddle points of $f$ above $u$ in $\mathring T$}\}.  
\end{equation}
  
\textbf{On the literature of Central limit theorems (CLT) for LKCs.} Recently, the asymptotic Gaussian fluctuations of the LKCs  of excursion sets (at any fixed threshold  $u$)  have  been proven in the case of a  Gaussian random field when the observation domain  grows to the whole Euclidean space. The interested reader is referred for instance to  \cite{bulinski2012central}, \cite{PHAM2013} for  CLTs for the volume. In \cite{KV18} and \cite{MULLER20171040} the authors prove a CLT for any Minkowski functional of excursion sets in the general framework of stationary Gaussian fields whose covariance function is decreasing fast enough at infinity 
(see also \cite{Spodarev13} for survey). On the same Gaussian setting for growing domain  and fast-decay covariance, in \cite{EstradeLeon16-AoP} and \cite{dibernardino2017} the  asymptotic fluctuations for  the  Euler-Poincar\'e characteristic is proved. In this central limit  theorems literature, the LKCs exhibit a variance asymptotically proportional to  the  volume of the growing domain. However, the Berry’s random wave model does not fall inside these studies. Indeed, the covariance function in \eqref{covariance} does not satisfy the assumptions required in these works, since it does not belong to $L^1(\mathbb{R}^2)$.  Instead of the class of  short-range covariance of the aforementioned literature, we work here with an intermediate-range covariance,  i.e, the Bessel function $J_0$.\\

For Berry's random wave model on $\mathbb R^2$, several results have been shown for specific geometric functionals. The nodal sets  (i.e., excursion sets at 0 level)   of Berry’s   random waves  in $\R^2$ have been studied in \cite{NourdinPeccatiRossi2017} in the    high-energy limit  setting  and their results rigorously confirm the asymptotic behaviour for the variances of the nodal length derived in \cite{Berry2002}. The obtained  asymptotics, which  can be  naturally reformulated in terms of the nodal statistics of a single random wave restricted to a compact domain diverging to the whole plane, show  that  the variance of the nodal length on a domain $T$ on  $\R^2$ is asymptotically proportional to  area($T$)$\log$(area($T$)) as $T$ grows up to $\R^2$ (see Theorem 1.1 in \cite{NourdinPeccatiRossi2017}). Complements of this theory can be found also in \cite{Vidotto}.   In \cite{Smutek} the author investigates the fluctuations of the nodal number of two independent real Berry Random Waves, with distinct energies. Furthermore, in  \cite{BCW20} critical points  for the planar Berry’s random wave model have been investigated, while the excursion area  can be derived by using the \textit{polyspectra} (stochastic terms in the Wiener chaos decomposition of integral functionals) in  \cite{GMT23} and the asymptotic behavior of a general class of functionals in  \cite{Maini2022}. In  \cite{Estrade2020AnisotropicGW} anisotropic random waves are considered in any dimension. 
   
In \cite{DalmaoEstradeLeon} the authors study the expected length  nodal lines of  Berry’s random waves model in $\R^3$ for a long-range  power law covariance model. They explicitly show how the decay of the covariance   impacts on the  unusual normalizing power of the volume of the growing domain in the behaviour of the asymptotic variance  (see Proposition 3.5 in \cite{DalmaoEstradeLeon}). In this long-ranged dependence, they also exhibit a non-Gaussian limit,  which is in hard contrast with the classical short-range covariance situations.
\\ 
  
 A strongly related problem is the study of the Gaussian Laplace eigenfunctions on the unit sphere $\Sphere^2$. Namely, for $\ell-$th energy level, such centred Gaussian random fields $\{T_\ell (x),\, x\in\Sphere^2\}$ have the covariance function
$$\mathbb E[T_\ell(x)T_\ell(y)]=P_\ell (\cos(\text{dist}(x,y))),$$
 where $P_\ell$ is the $\ell-$th Legendre polynomial and $\text{dist}(x,y)$ is the geodesic distance on the sphere, i.e., $dist(x,y)=\arccos(\langle x,y\rangle)$, where $\langle x,y\rangle$ is the standard scalar product in $\mathbb{R}^3$. For large $\ell$ the covariance function locally approximates $J_0$ by virtue of Hilb's asymptotics (\cite{Szego}, Theorem 8.21.12), and there is hope to relate results for $T_\ell$ to corresponding questions for the planar field $f$ in \eqref{covariance}. This is a special case of Berry's famous conjecture (see \cite{Berry1977}) stating that the high energy behaviour of Laplace eigenfunctions is universal across ``generic'' Riemannian manifolds. If proven, this bridge would be extremely useful, because the symmetry of $\Sphere^2$ has been fruitfully exploited for the study of geometric functionals in recent years. The asymptotic behaviour of the nodal length of random  spherical harmonics  is given  in   \cite{MRW20} and the CLT for shrinking caps  is established in   \cite{AnnaPaolaTodino}.  In \cite{CMW16b} the authors study the limiting distribution of critical points and extrema of random spherical harmonics, in the high energy limit.  The correlation between the total number of critical points of random spherical harmonics and the number of critical points with value in any real interval is studied in \cite{CT22}. In \cite{CMWa} a precise expression for the asymptotic variance of the Euler-Poincar\'e characteristic for excursion sets of random spherical harmonics is presented. A quantitative CLT in this high-energy limit  have been investigated in \cite{CM18}. Remarkably, for almost all levels the behaviour of the three LKC on $\Sphere^2$ is determined by $H_2(T_\ell)$, the second Wiener chaos component of the initial field $T_\ell$. More precisely (see \cite{domenico}), for $k=0,1,2$,
 \begin{equation}\label{lkc-sphere}
     \mathcal L_k(A(T_\ell, \Sphere^2, u))- \mathbb E[\mathcal L_k(A(T_\ell, \Sphere^2, u))] \stackrel{\ell\to \infty}{\sim} c_k(u)(\lambda_\ell)^{(2-k)/2}\int_{\mathbb S^2} H_2(T_\ell (x))dx,
 \end{equation}
 where $\lambda_\ell$ is the eigenvalue at the energy level $\ell$, and the explicit coefficients $c_k$ are such that $c_k(u)=0$ iff $u=0$ for $k=1,\,2$, and $c_0(u)=0$ iff $u\in\{-1,0,1\}$. If one knew by Berry's conjecture that Equation \eqref{lkc-sphere} also holds on the plane, the result of \cite{GMT23} would imply that the variance of $\varphi (f, T, u)$ for $u\notin\{-1,0,1\}$ and  $f$ in \eqref{covariance} is of order $\text{area}(T)^{3/2}$ in contrast to the short-range covariance case  (see the  aforementioned literature, e.g.,    \cite{EstradeLeon16-AoP}, \cite{KV18}, \cite{MULLER20171040}).\\

 A local version of Berry's conjecture has been proven in \cite{CH}, via a coupling between a random field on a manifold and a random field on its tangent space. The authors give an estimate for the expectation and variance of the nodal length and the number of critical points for a large class of manifolds. This result has been completed in \cite{Dierickx} with a CLT for the nodal length over shrinking balls. In \cite{CH} the authors also provide global estimates for the nodal length and critical points on compact manifolds. As for precise results on the plane confirming Berry's conjecture in the sense of establishing \eqref{lkc-sphere}, one can refer to \cite{NourdinPeccatiRossi2017} for the nodal length (see also \cite{igorsurvey} for an explicit comparison between $\mathbb R^2$ and $\mathbb S^2$), and, implicitly, in \cite{GMT23} for the area.\\

 Our work confirms the prediction of $\Var(\varphi (f, T, u))$ by Berry's conjecture, determines the precise normalising constant as well as a CLT, concludes thereby the study of Berry's random wave LKC on $\mathbb R^2$, and provides an instance of a change in behaviour compared to the short-range covariance case. Based on this result, we can find precise asymptotics for a randomly perturbed random wave field. The next section is dedicated to details and implications of our results.\\

 \textbf{Some conventions.} For the rest of the paper, we assume that all random objects are
defined on a common probability space $(\Omega, F, \Pbb)$, with $\E$ denoting expectation with respect to $\Pbb$. In the whole paper $\phi$ will be  the probability density function and $\Phi$  the cumulative distribution function of the standard Gaussian  random variable $\Normal(0,1)$.  Given two positive sequences $a_n$ and $b_n$, we write  $a_n\sim b_n$ if  $\lim_{n \to \infty} \frac{a_n}{b_n}=1$.  We use the symbol ${\overset{\mathcal L}{\longrightarrow}}$ to denote convergence in distribution. 
 
\subsection{Main results} 
Our first goal is to establish the Central Limit Theorem for the  modified Euler-Poincaré Characteristic of the excursion above $u\in\mathbb R$ in the Berry’s model in \eqref{covariance} in a growing domain $T_N=[-N,N]^2$, with $N$ a positive integer, with  $N \to \infty$ (see Theorem \ref{TCL}).  Secondly,  we define a class of  Berry’s mixture models and we study the obtained variances and the Gaussian fluctuations in this perturbed setting (see Theorem \ref{TCLperturbed}).\\ 

\textbf{Gaussian fluctuations of the EPC in the planar Berry’s model.} The desired asymptotic behaviour is obtained exploiting the $L^2(\Omega)$ expansion of modified Euler characteristic into Wiener chaoses, which are orthogonal spaces spanned by Hermite polynomials. First of all, we recall that the Hermite polynomials $H_q(x)$ are defined by $H_0(x)=1$, and for $q=2,3,\dots$  
\begin{equation}\label{Hermite}
H_q(x)=(-1)^q \frac{1}{\phi(x)} \frac{d^q\phi(x)}{dx^q},
\end{equation}
with $x \in \R$. We consider the Wiener chaos expansion of $\mathcal{\varphi}(f,T_N,u)$ (see, e.g., \cite[Section 1.2]{EstradeLeon16-AoP} and \cite{CM18}).
\begin{align}\label{Proj}
\mathcal{\varphi}(f,T_N,u)=\sum_{q=0}^{\infty} \mathcal{\varphi}(f,T_N,u)[q],
\end{align}
where $\mathcal{\varphi}(f,T_N,u)[q]$ denotes the projection of $\mathcal{\varphi}(f,T_N,u)$ on the $q$-order chaos component, that is, the space generated by the $L^2$-completion of linear combinations of the form 
$$H_{q_1}(\xi_1) \cdot H_{q_2}(\xi_2) \cdots H_{q_k}(\xi_k), \hspace{1cm} k \ge 1,$$ 
with $q_i \in \mathbb{N}$ such that $q_1+\cdots+q_k=q$, and $(\xi_1, \dots, \xi_k)$  real standard  Gaussian vector. 

It results that (after centring) a single term dominates the $L^2(\Omega)$ expansion in (\ref{Proj}), that is, the projection into the second chaos (see Proposition 
    \ref{th:variance}).  
 As explained above, this was expected since random spherical harmonics exhibit the same behaviour as Berry's random wave (see, e.g.,  \cite{igorsurvey}). 
 
\begin{theorem}\label{TCL}
Recall that $|T_N|=(2N)^2$. Let $f: \mathbb{R}^2 \to \mathbb{R}$ 
 be the isotropic standard  Berry's Gaussian random field  with covariance function in \eqref{covariance}.  Denote by $\lambda_f=\frac{1}{2}$ the second spectral moment of $f$. Then for any $u$ in $\R$, with $u \notin \{0,1, -1\}$, it holds that 	$$\frac{\varphi(f,T_N,u)-\mathbb{E}[\varphi(f,T_N,u)]}{\sqrt{  \frac{1}{2}\mathcal L_{1}(T_N)  \mathcal L_{2}(T_N)}}   \,\underset{N\to \infty}{\overset{\mathcal L}{\longrightarrow}}~\Normal(0,V(u)),$$   where
\begin{equation} \label{meanEC}
\E[\varphi(f,T_N,u)]=\mathcal L_{2}(T_N)\,(2\pi)^{-\frac{3}{2}}\,\lambda_f \,u \,e^{-u^2/2},
\end{equation} and $\Normal(0,V(u))$ stands for the centered Gaussian distribution with finite and strictly positive variance, 
 \begin{align}   \label{eqVu}   
   V(u)  = \upsilon(u)\,\frac{1}{8 \pi^{2}}   \left(\frac{1-\sqrt{2}}{3} +\ln(1 + \sqrt{2})\right),
    \end{align}
with \begin{equation}\label{g}
    \upsilon(u):=   \phi \left(u\right)^2 u^2\left(u^2-1\right)^2.\end{equation}
\end{theorem}

Section \ref{BerrySection} is devoted to the proof of Theorem  \ref{TCL}.

\begin{remark}\label{proofstategy}
{Although this statement, and in particular the dependence of the variance on the level $u$, is quite similar to the spherical case studied in \cite{CM18} and \cite{CMWa}, our proof strategy is different. In the above references, the authors rely on the Kac-Rice formula type statement to calculate the variance and then compute the second chaos variance that turns out to be the same in the high-frequency asymptotics, allowing them to ignore all the other terms for the CLT. In contrast to this, we compute the order of the variance of each chaos component separately, which requires in particular a careful treatment of the first chaos. This allows us to see ``from below'' that the variance is dominated by a single component. Moreover, contrary to the spherical case, we cannot rely on a representation of the field and its derivatives as a linear combination of independent random variables. Instead, we show that there is perfect asymptotic correlation between the second chaos component and its part containing only the contribution of the initial field $f$, which allows us to obtain a CLT as an application of the spectral CLT recently shown by \cite{Maini2022}.}
\end{remark}

\textbf{Gaussian fluctuations of the EPC in a class of perturbed Berry’s models.}  Our second aim is to establish the Central Limit Theorem for the  modified Euler Poincaré Characteristic $\varphi$ in the case of a \textit{perturbed  Berry’s model}.  

We will study Berry mixture model $\{f_\Lambda(x)\,:\,x\in T\}$, with $T$ 
 fixed compact domain   in $\R^2$, that have the following general structure:
\begin{equation} \label{Perturbation}
f_\Lambda(x)=g(f(x),\Lambda), \, \mbox{ with } \{f(x)\} \perp \Lambda, \, x\in T, 
\end{equation}
where  $\Lambda$ is a shape variable with suitable  properties and $g$ is a certain link function,  such that  $g(\cdot,\lambda)$ is strictly increasing for any $\lambda$. 

Specifically, location mixtures arise for $g(w,\lambda)=\lambda w$ (called \emph{Gaussian scale mixture}), whereas  scale mixtures are obtained by setting   $g(w,\lambda)=w+\lambda$  (\emph{Gaussian location mixtures}).  In the first case $\Lambda > 0$ can be viewed as  a random standard deviation parameter embedded in the Gaussian 
 random field $f$, the second one as a  random mean  Gaussian model. 
   
Such processes have received considerable attention in the recent spatial statistics literature due to their ability to
account for asymmetric lower and upper tails, and for extremal dependence that is stronger than in the purely Gaussian case
(see \emph{e.g.}, \cite{Krupskii.al.2018}). In spatial extreme-value theory, they lead to extensively studied limit processes of Brown--Resnick type~(see \emph{e.g.}, \cite{Kabluchko.Schlather.deHaan.2009,Wadsworth.Tawn.2014,Dombry.al.2016}).
 
This setting allows us to introduce the inverse function $h(\cdot,\lambda)$ of $g(\cdot,\lambda)$ so that from \eqref{EXCSET} we have the following almost sure identity for the excursion set  above the  level $u$:   
\begin{equation}\label{PertSET}
A(f_\Lambda, T_N, u)=  A(f, T_N, h(u,\Lambda)). 
\end{equation} 

Using the Central Limit Theorem \ref{TCL} we derive the asymptotic behaviour of the variance of $\varphi(f_\Lambda, T_N, u)$ in the perturbed Berry’s model introduced in \eqref{Perturbation}  and the Gaussian fluctuations as $T_N \nearrow  \mathbb{R}^2$, i.e., $N \to \infty$ (see Theorem \ref{TCLperturbed} below).  For any function $\zeta$ such that $\mathbb{E}[\zeta(A(f_\Lambda,T_N,u))]<\infty$, we can write
$$\mathbb{E}[\zeta(A(f_\Lambda,T_N,u))]=\mathbb{E}[\mathbb{E}[\zeta(A(f,T_N,h(u,\Lambda)))|\Lambda]].$$

\begin{theorem}\label{TCLperturbed} 
Recall that $|T_N|=(2N)^2$. Let $f: \mathbb{R}^2 \to \mathbb{R}$ 
 be the isotropic standard  Berry's Gaussian random field  with covariance function in \eqref{covariance}. Let 
 $\{f_\Lambda(x)\,:\,x\in T_N\}$ be the Berry mixture model in  \eqref{Perturbation}. 
Denote by $\lambda_f=\frac{1}{2}$ the second spectral moment of $f$.   Then, for any $u \in \R$,  such that $\mathbb{E}\left[\upsilon\left(h(u,\Lambda)\right)\right] \neq 0$, with $\upsilon$ as in \eqref{g}, 
it holds that 	
    $$\frac{\mathcal{\varphi}(f,T_N,h(u,{\Lambda}))-\mathbb{ E}[\mathcal{\varphi}(f,T_N,h(u,{\Lambda}))]}{\sqrt{\frac{1}{2}\mathcal L_{1}(T_N)  \mathcal L_{2}(T_N)}}    \,\underset{N\to \infty}{\overset{\mathcal L}{\longrightarrow}}~\mathcal{N}(0,\,V_\Lambda(u)),$$
 where  
\begin{equation}\label{meanECperturbed}\mathbb{E}[\varphi(A(f_\Lambda,T_N,u))]= \mathcal{L}_2(T_N) \, \lambda_f  \, \mathbb{E}[\rho_2(h(u,\Lambda))],\end{equation} 
  $\rho_2(u):=(2\pi)^{-3/2} e^{-u^2/2} H_1(u)= (2\pi)^{-1/2} \phi(u)\,u$ and 
\begin{equation}\label{variancePerturbed}
V_\Lambda (u)= \mathbb{ E}\left[\upsilon\left(h(u,\Lambda)\right)\right] \frac{1}{8\pi^2} \left(\frac{1-\sqrt{2}}{3}+\ln(1+\sqrt{2})\right),
\end{equation}
where  $h(\cdot,\lambda)$ is the inverse function   of $g(\cdot,\lambda)$. 
    \end{theorem}

Section \ref{sec:proof main} is devoted to the proof of Theorem \ref{TCLperturbed}. 
Obviously, Theorem \ref{TCL} can be seen as a specific case of the Central Limit Theorem \ref{TCLperturbed}   for  Berry's mixture models where  $h(u, \Lambda)$  is an almost surely constant random variable.

The paper is organized as follows. In Section \ref{BerrySection} and Section \ref{sec:proof main}
we prove Theorems \ref{TCL} and  \ref{TCLperturbed}, respectively. In Section \ref{NUmeric} we present explicit computations for some parametric models for the perturbed Berry’s random wave and numerical studies.  Finally,
technical lemmas exploited in the proof of Theorem \ref{TCL} are collected in Appendix \ref{sec:technical}.
 
\section{Proof of Central Limit Theorem \ref{TCL}}\label{BerrySection}
 
 The first moment in \eqref{meanEC}  can be obtained via the Rice formulas for the factorial moments of the number of local maxima above $u$,   the number of local minima above $u$ and  the number of local saddle points above $u$  (see for instance \cite{AT07} Chapter 11 or \cite{AW} Chapter 6).  
 
In Proposition \ref{th:variance} below, we compute the asymptotic variance of the EPC, exploiting the Wiener chaos decomposition in \eqref{Proj}, and conclude that the second chaos is the leading term of the series expansion. Then, in Proposition \ref{Corr}, we prove that there is asymptotically  full correlation between the Euler-Poincaré characteristic and the second chaotic projection. By combining  the variance in Proposition \ref{th:variance}  and this full correlation result in Proposition \ref{Corr}, the Central Limit Theorem \ref{TCL} follows.

\begin{proposition}
    \label{th:variance} Let $f:\Omega \times \mathbb{R}^2 \to \mathbb{R}$ 
 be the isotropic  standard  Berry's Gaussian random field  with covariance function in \eqref{covariance}.   Then for any $u$ in $\R$,    
	as $N\to \infty$, we have

$$\mathcal{\varphi}(f,T_N,u)-\mathbb{E}[\mathcal{\varphi}(f,T_N,u)]=\mathcal{\varphi}(f,T_N,u)[2] + o(\mathcal{\varphi}(f,T_N,u)[2])$$ and then 
$${(2\,N)^{-3}}\Var(\varphi(f,T_N,u))=  \upsilon(u) \frac{1}{8 \pi^2}   \left(\frac{1-\sqrt{2}}{3} +\ln(1 + \sqrt{2})\right) +O\left(\frac{1}{\sqrt{N}}\right),$$
where $\upsilon(u)$ is defined in (\ref{g}).
\end{proposition}
The proof of Proposition \ref{th:variance} is postponed to Section \ref{preuveProp21}. 
    
\begin{proposition}\label{Corr}
Let
    $$\mathcal{S}(f,N):= \int_{[-N,N]^2} H_2(f(x))dx.$$ 
then we have
$$\left|\operatorname{Corr} \left(\mathcal{\varphi}(f,T_N,u)[2], \,\mathcal{S}(f,N)\right)\right|\stackrel{N\to\infty}{\to}1.$$
\end{proposition}
The proof of Proposition \ref{Corr} is postponed to Section \ref{PreuveProp22}.  
\begin{proof}[ \textbf{Proof of Theorem \ref{TCL}}]
   
In Proposition \ref{th:variance} we establish that, as $N \to \infty$, and $u\ne\{0,-1,1\}$
$$\mathcal{\varphi}(f,T_N,u)-\mathbb{E}[\mathcal{\varphi}(f,T_N,u)]=\mathcal{\varphi}(f,T_N,u)[2] + o(\mathcal{\varphi}(f,T_N,u)[2]).$$


In view of Proposition \ref{Corr} the Euler-Poincaré characteristic is fully correlated in the limit to $\mathcal{S}(f,N)$, i.e.   
\begin{equation}\label{correlation1}
\lim_{N\to \infty}\text{Corr}(\mathcal{\varphi}(f,T_N,u), \mathcal{S}(f, N))=\lim_{N\to \infty}\frac{\text{Cov}(\mathcal{\varphi}(f,T_N,u), \mathcal{S}(f,N))}{\sqrt{ \text{Var}(\mathcal{\varphi}(f,T_N,u)) \text{Var}(\mathcal{S}(f,N))}} = 1.  
\end{equation}
Moreover, 
$$\lim_{N\to \infty} \frac{Var(\mathcal{\varphi}(f,T_N,u)[2])}{Var(\mathcal{\varphi}(f,T_N,u))}=1+o(1).$$ Then, by using the Wasserstein distance $d_W$ {and denoting $Z$ a standard Gaussian random variable},  we have
\begin{eqnarray*}
&&d_W\left(\frac{\mathcal{\varphi}(f,T_N,u)-\mathbb{ E}[\mathcal{\varphi}(f,T_N,u)]}{\sqrt{Var(\mathcal{\varphi}(f,T_N,u))}},Z\right) \\&&\quad \leq d_W \left(\frac{\mathcal{\varphi}(f,T_N,u)-\mathbb{ E}[\mathcal{\varphi}(f,T_N,u)]}{\sqrt{Var(\mathcal{\varphi}(f,T_N,u))}},  \frac{\mathcal{\varphi}(f,T_N,u)[2]}{\sqrt{Var(\mathcal{\varphi}(f,T_N,u))}}  \right) + d_W\left(\frac{\mathcal{\varphi}(f,T_N,u)[2]}{\sqrt{Var(\mathcal{\varphi}(f,T_N,u))}} , Z\right).
\end{eqnarray*}

Now we can bound the first term by

\begin{eqnarray*}
  d_W \biggl(\frac{\mathcal{\varphi}(f,T_N,u)-\mathbb{ E}[\mathcal{\varphi}(f,T_N,u)]}{\sqrt{Var(\mathcal{\varphi}(f,T_N,u)))}}, &&  \frac{\mathcal{\varphi}(f,T_N,u)[2]}{\sqrt{Var(\mathcal{\varphi}(f,T_N,u)}}  \biggl) \\&&\quad \leq \sqrt{\mathbb{ E}\left[ \frac{\mathcal{\varphi}(f,T_N,u)-\mathbb{ E}(\mathcal{\varphi}(f,T_N,u)))-\mathcal{\varphi}(f,T_N,u)[2]}{\sqrt{Var(\mathcal{\varphi}(f,T_N,u)))}}\right]^2},
\end{eqnarray*}
which tends to 0 as $N\to \infty$ (see    Proposition \ref{th:variance}).
Hence, we have that

\begin{eqnarray*}
d_W\left(\frac{\mathcal{\varphi}(f,T_N,u)[2]}{\sqrt{Var(\mathcal{\varphi}(f,T_N,u))}} , Z\right) = d_W(S(f,N),Z)+o(1).
\end{eqnarray*}

Finally,  $d_W(S(f,N),Z) \to 0$, applying the Central Limit Theorem in \cite[Theorem 2]{Maini2022}. Indeed, the authors show that the spectral condition required for weak convergence is satisfied for the planar Berry random field. The only thing remaining to show is a condition on the Fourier transform $\mathcal{F} $ of the rescaled domain of integration $D:=[-1,\, 1]^2$, namely
$|\mathcal{F}[1_D](x)|= O\left(\frac{1}{\|x\|}\right)$
as $\|x\|\to \infty$. We have
$|\mathcal{F}[1_D](x)| = \frac{4 \sin(x_1)\sin(x_2)}{x_1x_2},$
which satisfies the necessary condition.
\end{proof}

\subsection{Proof of Proposition \ref{th:variance}}\label{preuveProp21}

To prove Proposition \ref{th:variance} we first need to establish the $L^2$ expansion of $\varphi(f,T_N, u)$ into Wiener
chaoses. To this aim we recall a result by Estrade and Léon \cite[Proposition 1.2]{EstradeLeon16-AoP} in Lemma \ref{lem:EL} below, which allows to write the modified EPC as the limit of an integral in $L^2.$ This representation permits the Wiener chaos expansion (after normalising the vector of first and second derivatives to have unit variance). To study the behaviour of different chaos components we need to compute all the covariances between $f, \partial_i f, \partial_{ii}f$ for all $i=1,2$. These calculations are made in Lemma \ref{covs}. The asymptotic behaviour of all chaos components in large domain (i.e., $N\to \infty$) can be found in Lemma \ref{leading}, \ref{lem:5}, \ref{lem:oN3}, \ref{lem:dominant} and Remark \ref{1st-tech}, collected in  Appendix \ref{sec:technical}. From these results we establish that the leading terms of the series expansion belong all to the second chaos and they behave as $N^3$; all the other terms are $o(N^3)$. Then the asymptotic variance in Proposition \ref{th:variance} is given by the variance of the dominant terms and we conclude the proof.\\
 
The following result is a close adaptation of \cite[Proposition 1.2]{EstradeLeon16-AoP} to our setup.
  
\begin{lemma}\label{lem:EL}
The following convergence holds almost surely and in $L^2(\Omega)$  
\begin{equation}\label{L2integral}
    \varphi (f, T_N, u)= \lim_{\varepsilon \to 0^{+}}\int\limits_{[-N,N]^2} \begin{vmatrix}
       \partial_{11}f(x) &  \partial_{12}f(x)\\ \partial_{12}f(x) & \partial_{22}f(x)
   \end{vmatrix}  {1}_{[u,\infty)}\left(-(\partial_{11}f(x)+\partial_{22}f(x))\right)\delta_\varepsilon (\nabla f(x))dx.
\end{equation} 
\end{lemma}

\begin{proof}
First, note that the integral appearing in the statement of the lemma can be written through the Wiener chaos decomposition.
More precisely, following similar notation as in \cite[Section 1.2]{EstradeLeon16-AoP},
we define
\begin{equation}
\varphi(\varepsilon, T_N):= \int_{T_N}\begin{vmatrix}
\partial_{11}f(x) &  \partial_{12}f(x)\\ \partial_{12}f(x) & \partial_{22}f(x)
\end{vmatrix}  {1}_{[u,\infty)}\left(f(x))\right)\delta_\varepsilon (\nabla f(x))dx,
\end{equation}
which can be rewritten as the r.h.s in \eqref{L2integral} since $f$ is the Laplacian eigenfunction with eigenvalue $1$. Then
we consider the map $G_\varepsilon$ defined on $\mathbb{R}^5$ by
$$G_\varepsilon(x,y,z)= \delta_\varepsilon(x) \tilde{det}(y) 1_{[u,\infty)}(y), \quad (x,y) \in \mathbb{R}^2 \times \mathbb{R}^3,$$
and the map
$$G_u(y,z)= \tilde{det}(y)1_{[u,\infty)}(y), \quad y \in \mathbb{R}^3.$$
We denote by $\Sigma^X$ the covariance matrix of the 5-dimensional Gaussian vector $X(t)=(\nabla f(t), \nabla^2 f(t))$ and 
we consider $\Xi$ a $5\times 5$ matrix such that $\Xi^T\Xi=\Sigma^X$. 
We can thus write, for any $t \in \mathbb{R}^2$, $X(t)=\Xi \, Y(t)$ with $Y(t)$ a 5-dimensional standard Gaussian vector.  The matrix $\Xi$ can be factorizes into $\begin{pmatrix}
    \Xi_1 & 0\\
    0&\Xi_2
\end{pmatrix},$ where  $\Xi_1=\sqrt{\lambda_f}I_2$ with $\lambda_f$ the second spectral moment of $f$.
For $y=(\underline{y}, \overline{y}) \in \mathbb{R}^5=\mathbb{R}^2\times \mathbb{R}^{3}$, we define 
$$\tilde{G}_{\varepsilon}(y)=G_{\varepsilon}
(\Xi y)= \delta_{\varepsilon}(\Xi_1 \underline{y}) G_u(\Xi_2 \overline{y})=\delta_\varepsilon \circ \Xi_1(\underline{y}) G_u \circ \Xi_2(\overline{y}).$$  
For 
$n=(n_1,\dots, n_5) \in \mathbb N^5$ and  $y \in \mathbb{R}^5=\mathbb{R}^{2}\times \mathbb{R}^{3}$,
$${\overline{H}}_{{n}}(y)= \prod_{1\leq j\leq 5}  H_{n_j}(y_j),$$ 
with $H_{n_j}$ the Hermite polynomials recalled in \eqref{Hermite}.
Now we can write $$G_\varepsilon (y)= \sum_{q=0}^{\infty} \sum_{n \in N^5; |n|=q} \langle\tilde{G}_\varepsilon, \overline{H}_n \rangle  \overline{H}_n(y),$$ where the Hermite coefficients $\langle\tilde{G}_\varepsilon, \overline{H}_n \rangle= \mathbb{E}[{G}_\varepsilon(X) \overline{H}_n (X) ]$ can be factorized as $\langle\tilde{G}_\varepsilon, \overline{H}_n \rangle = \langle \delta_{\varepsilon} \circ \Xi_1, \overline{H}_{\underline{n}} \rangle  \langle G_u \circ \Xi_2, \overline{H}_{\overline{n}} \rangle $ given respectively by
$$ \langle \delta_{\varepsilon} \circ \Xi_1, \overline{H}_{\underline{n}} \rangle = \frac{1}{\underline{n}!} \int_{\mathbb{R}^2} \delta_{\varepsilon} (\sqrt{\lambda_f}y) \overline{H}_{\underline{n}} (y) \phi_2(y) \, dy$$ and
$$\langle G_u \circ \,\Xi_2, \overline{H}_{\Bar{n}} \rangle :=\frac{1}{\Bar{n}!} \int_{\mathbb R ^{3}} G_u(\Xi_2,\,z){\overline{H}}_{\bar{n}}(z) \phi_{3}(z)\,dz,$$
with $\phi_m,$ for $m=2,4$ the standard Gaussian density on $\mathbb{R}^m$.
We take the limit as $\varepsilon \to 0 $ to obtain the expansion of the modified EPC. We have that
$$\langle \delta_{\varepsilon} \circ \Xi_1, \overline{H}_{\underline{n}} \rangle  \to_{\varepsilon \to 0} \frac{1}{n!} (2\pi \lambda_f)^{-1} \overline{H}_{\underline{n}}(0). $$ Then, 
for $n=(\underline{n},\overline{n}) \in \mathbb{R}^2\times \mathbb{R}^{3} $, the Hermite coefficients $\langle G_u,\overline{H}_{{n}} \rangle$ are given by $$\langle G_u,\overline{H}_{{n}} \rangle =\frac{1}{\underline{n}!}(2\pi\lambda_f)^{-1} {\overline{H}}_{\underline{n}}(0) \langle G_u \circ\, \Xi_2, \Bar{n} \rangle.$$
By \cite[Proposition 1.3]{EstradeLeon16-AoP} the expansion of the modified EPC is finally given by
$$\sum_{\substack{q=0}}^\infty \,\,\sum_{{n}\in\mathbb N^5,\\ |n|=q} \langle G_u,\, \overline{H_{{n}}}\rangle h_{n},$$
where $|n|=n_1+\dots+n_5$ and
 $$h_{{n}}=\int_{[-N,N]^2} \overline{H}_{{n}}(f, \nabla f, \nabla ^2 f)(x) \,dx.$$
We refer to \cite[Section 1.2]{EstradeLeon16-AoP} for more details. 

It follows that, the chaotic decomposition in (\ref{Proj}) can be written as
\begin{equation}\label{expansion}
    \mathcal{\varphi}(f,T_N,u)-\mathbb{E}[\mathcal{\varphi}(f,T_N,u)]=\sum_{q=1}^{\infty} \mathcal{\varphi}(f,T_N,u)[q]=\sum_{\substack{q=1}}^\infty \sum_{\substack{{n}\in\mathbb N^5,\\ |{n}|=q}} \langle G_u,\, \overline{H_{{n}}}\rangle h_{{n}}.
    \end{equation}
    \end{proof}

Now we can analyse the behaviour of different chaos components $h_n$. To this aim, we will need to consider the polyspectra of the first and second order derivatives of $f$. The following technical result helps evaluate their covariances.
\begin{lemma}
    \label{covs}
    We have for $x,\, y \in \mathbb R$ such that $x=(0,0)$, $y=(0,r)$:
    \begin{eqnarray*}
    &&g_1(r):=\mathbb E [\partial_1 f(x)\partial_1 f(y)] = \frac{J_1(r)}{r}, \cr
    && g_2(r):=\mathbb E [\partial_2 f(x)\partial_2 f(y)] = \frac{J_0(r)-J_2(r)}{2},\cr
    && g_3(r):= \mathbb E [\partial_{11} f(x)\partial_{2} f(y)]=\mathbb E [\partial_1 f(x)\partial_{12} f(y)] = \frac{J_0(r)-J_2(r)}{2r}-\frac{J_1(r)}{r^2}= -\frac{J_2(r)}{r},\cr
    &&g_4(r):= \mathbb E [\partial_2 f(x)\partial_{22} f(y)] = \frac{J_3(r)-3J_1(r)}{4},\cr
    &&g_5(r):= \mathbb E [\partial_{11} f(x)\partial_{11} f(y)] = \frac{3(J_2(r)-J_0(r))}{2r^2}+\frac{3J_1(r)}{r^3}=\frac{3J_2(r)}{r^2},\cr
    &&g_6(r):= \mathbb E [\partial_{11} f(x)\partial_{22} f(y)] = \mathbb E [\partial_{12} f(x)\partial_{12} f(y)]=\frac{3J_1(r)-J_3(r)}{4r}+\frac{J_0(r)-J_2(r)}{r^2}-\frac{2J_1(r)}{r^3},\cr
    && g_7(r):=\mathbb E [\partial_{22} f(x)\partial_{22} f(y)] = \frac{3J_0(r)-4J_2(r)+J_4(r)}{8}=-\frac{(r^2-3)J_2(r)}{r^2},
    \end{eqnarray*}
    while all the other combinations are uncorrelated. The symmetric expressions for $x=(0,r)$, $y=(0,0)$ are identical up to a possible multiplication by $-1$.
\end{lemma}

\begin{proof}[Proof of Lemma \ref{covs}]
	We compute, for instance, for $g_1$:
	\begin{eqnarray*}
		&&\mathbb E [\partial_1 f(x)\partial_1 f(y)] = \partial_{x_1} \partial_{y_1}J_0\left(\left\| \begin{pmatrix}
			x_1 \\ 
			x_2
		\end{pmatrix}-  \begin{pmatrix}
			y_1 \\ 
			y_2
		\end{pmatrix}\right\|\right) = \partial_{y_1}J'_0\left(\left\|x-y\right\|\right)\frac{x_1-y_1}{\left\|x-y\right\|}\\
		&& \qquad \frac{-\|x-y\|-(y_1-x_1)\frac{x_1-y_1}{\|x-y\|}}{\|x-y\|^2}J'_0(\|x-y\|)+ \frac{(y_1-x_1)(x_1-y_1)}{\|x-y\|^2}J''_0(\|x-y\|).
	\end{eqnarray*} 
	This equals $\frac{J_1(r)}{r}$ for both $x=(0,0)$, $y=(0,r)$ and $x=(0,r)$, $y=(0,0)$.\\
	As for the symmetry property, note that all the derivatives have the form
	$$h(x_1-y_1, x_2-y_2, \|x-y\|),$$
	and we can see by induction that the dependence on $(x_2-y_2)$ is polynomial, either even or odd, depending on the number of the second component derivatives. Therefore, going from $x=(0,0)$, $y=(0,r)$ to $x=(0,r)$, $y=(0,0)$ would not change the final result if the polynomial is even, and would change only the sign if the polynomial is odd.
\end{proof}

In order to obtain the expansion (\ref{expansion}) of the modified EPC in terms of Hermite polynomials, we also need to normalise the vector of first and second order derivatives. First, let us denote
\begin{eqnarray*}
	&& Y_i := \widetilde{\partial_if(x)}:=\frac{\partial_if(x)}{\sqrt{\operatorname{Var}(\partial_if(x))}}= \partial_if(x) \sqrt{2},\\
	&& Y_3 := \widetilde{\partial_{11}f(x)}:=\frac{\partial_{11}f(x)}{\sqrt{\operatorname{Var}(\partial_{11}f(x))}}= \partial_{11}f(x) \sqrt{\frac{8}{3}},\\
	&& Y_4 := \widetilde{\partial_{12}f(x)}:=\frac{\partial_{12}f(x)}{\sqrt{\operatorname{Var}(\partial_{12}f(x))}}= \partial_{12}f(x) \sqrt{8},\\
	&& Y_5 := \widetilde{\partial_{22}f(x)}:=\frac{\partial_{22}f(x)}{\sqrt{\operatorname{Var}(\partial_{22}f(x))}}= \partial_{22}f(x) \sqrt{\frac{8}{3}},
\end{eqnarray*}
where $i=1,2$. Next, by considering the limits of $g_3$, $g_4$, and $g_6$ in Lemma \ref{covs} as $r\to 0$, we see that
$$\mathbb E[Y_3Y_5]=\frac{1}{3},$$
while the other variables are uncorrelated. By setting $Z:=\frac{3}{2\sqrt{2}}Y_3-\frac{1}{2\sqrt{2}}Y_5$, we obtain an i.i.d. vector $(Y_1, Y_2, Z,Y_4,Y_5)$, and with this notation we can write (\ref{L2integral}) as
\begin{eqnarray*}
	&& \begin{vmatrix}
		\partial_{11}f(x) &  \partial_{12}f(x)\\ \partial_{12}f(x) & \partial_{22}f(x)
	\end{vmatrix}  {1}_{[u,\infty)}\left(-(\partial_{11}f(x)+\partial_{22}f(x))\right)\delta_\varepsilon (\nabla f(x))\\
	&& = \left(\frac{3}{8}\widetilde{\partial_{11}f(x)} \widetilde{\partial_{22}f(x)}- \frac{1}{8}\widetilde{\partial_{12}f(x)}^2 \right) 1_{[u,\infty)}\left(-\sqrt{\frac{3}{8}}\left(\widetilde{\partial_{11}f(x)}+\widetilde{\partial_{22}f(x)}\right)\right)\\
	&&\quad \times \delta_\varepsilon \left(\frac{1}{\sqrt{2}} \widetilde{\partial_{1}f(x)}\right)\delta_\varepsilon \left(\frac{1}{\sqrt{2}} \widetilde{\partial_{2}f(x)}\right)\\
	&&=\left(\frac{3}{8}\left(\frac{\sqrt{8}}{3}Z+\frac{1}{3}Y_5\right)Y_5- \frac{1}{8}Y_4^2 \right) 1_{(-\infty ,-u]}\left(\frac{1}{\sqrt{3}} Z+\frac{\sqrt{2}}{\sqrt{3}}Y_5\right)\\
	&&\quad \times  \delta_\varepsilon \left(\frac{1}{\sqrt{2}} Y_1\right)\delta_\varepsilon \left(\frac{1}{\sqrt{2}} Y_2\right).
\end{eqnarray*}
Now this expression can be expanded in terms of Wiener chaoses. Note that it is isotropic (although individual terms need not be), and therefore, when computing the covariance in two different point $x$ and $y$, we can fix the coordinates at $x=(0,0)$, $y=(0,r)$ once and for all.

In Lemmas \ref{leading}, \ref{lem:5}, \ref{lem:oN3}, \ref{lem:dominant} (proved in Appendix \ref{sec:technical}) we compute the covariances that arise from the chaos decomposition. Their results imply that the only summands contributing to the leading order are given by some terms of the second chaos, namely by $H_2(Y_2)$, $H_2(Y_5)$, $H_2(Z)$, $H_1(Y_2) H_1(Y_5)$, $H_1(Y_2) H_1(Z)$, and $H_1(Z) H_1(Y_5)$.
Let us call the respective coefficients with which they appear in the chaos decomposition $c_{2}$, $c_{5}$, $c_{z}$, $c_{25}$, $c_{2z}$, and $c_{z5}$. The precise expressions for these coefficients are given in Lemma \ref{coefficients}. Then,
let 
\begin{equation}\label{D2}
    \mathcal{D}_2 (f,N):=\int_{[-N,N]^2} \frac{1}{2}c_2 H_2(Y_2(x))+\frac{1}{2} c_5 H_2(Y_5(x))+\frac{1}{2} c_z H_2(Z(x))+ c_{z5}H_1(Z(x))H_1(Y_5(x))dx
    \end{equation}
denote the leading term of the second chaotic projection of the Wiener chaos expansion of the modified EPC. We have
$$\varphi(f,T_N,u)[2]= \mathcal{D}_2(f,N)+O(N^2\sqrt{N})$$ 
and 
$$\mathcal{\varphi}(f,T_N,u)-\mathbb{E}[\mathcal{\varphi}(f,T_N,u)]=\mathcal{\varphi}(f,T_N,u)[2] + o(\mathcal{\varphi}(f,T_N,u)[2]).$$

It follows that the variance $V:=\Var(\varphi(f,T_N,u)-\mathbb{E}[\mathcal{\varphi}(f,T_N,u)])$ is asymptotically equivalent to
\begin{eqnarray*}
   &&\mathbb{E}[\mathcal{D}_2(f,N)^2]= \\&&\mathbb E \left[\left(\int_{[-N,N]^2} \frac{1}{2}c_2 H_2(Y_2(x))+\frac{1}{2} c_5 H_2(Y_5(x))+\frac{1}{2} c_z H_2(Z(x))+ c_{z5}H_1(Z(x))H_1(Y_5(x))dx\right)^2\right],
   \end{eqnarray*}
	  where the coefficients $c_2,c_5,c_z,c_{z5}$ are given in Lemma \ref{coefficients}.
	Denoting $r:=\|x-y\|$ and plugging in all the normalisations defined previously, we get
 \begin{eqnarray}\label{sum}
	  V \sim	&& \int_{[-N,N]^2}\int_{[-N,N]^2} \frac{1}{2}c_2^2 (2g_2(r))^2 + \frac{1}{2}c_2c_5 \left(\frac{4}{\sqrt{3}}g_4(r)\right)^2 +\frac{1}{2}c_2c_z \left(\frac{4}{\sqrt{3\cdot 8}}g_4(r)\right)^2
  \nonumber \cr
		&& + c_2c_{z5}\left(-\frac{4}{\sqrt{3\cdot 8}}g_4(r) \cdot \frac{4}{\sqrt{3}}g_4(r)\right) 
	 + \frac{1}{2}c_5c_2 \left(\frac{4}{\sqrt{3}}g_4(r)\right)^2 + \frac{1}{2}c_5^2 \left(\frac{8}{3}g_7(r)\right)^2 \nonumber\cr &&+ \frac{1}{2}c_5c_z \left(-\frac{\sqrt{8}}{3}g_7(r)\right)^2 + c_5c_{z5}\left(-\frac{\sqrt{8}}{3}g_7(r) \cdot \frac{8}{3}g_7(r)\right)+\frac{1}{2}c_zc_2 \left(\frac{4}{\sqrt{3\cdot 8}}g_4(r)\right)^2 \nonumber \cr &&+ \frac{1}{2}c_zc_5 \left(-\frac{\sqrt{8}}{3}g_7(r)\right)^2 + \frac{1}{2}c_z^2 \left(\frac{1}{3}g_7(r)\right)^2 + c_zc_{z5}\left(-\frac{\sqrt{8}}{3}g_7(r) \cdot \frac{1}{3}g_7(r)\right) \nonumber \cr
		&& + c_{z5}c_2 \left(-\frac{4}{\sqrt{3\cdot 8}}g_4(r) \cdot \frac{4}{\sqrt{3}}g_4(r)\right) +c_{z5}c_5 \left(-\frac{\sqrt{8}}{3}g_7(r) \cdot \frac{8}{3}g_7(r)\right) \nonumber \\
		&& + c_{z5}c_z \left(-\frac{\sqrt{8}}{3}g_7(r) \cdot \frac{1}{3}g_7(r)\right) + c_{z5}^2 \left(\frac{8}{3}g_7(r)\frac{1}{3}g_7(r)+ \left(-\frac{\sqrt{8}}{3}g_7(r)\right)\left(-\frac{\sqrt{8}}{3}g_7(r)\right)\right)dx dy \nonumber \\ .
  \end{eqnarray}
	Since
	$$\int_{[-N,N]^2}\int_{[-N,N]^2}g_2^2(r)dxdy \sim \int_{[-N,N]^2}\int_{[-N,N]^2}g_4^2(r)dxdy \sim \int_{[-N,N]^2}\int_{[-N,N]^2} g_7^2(r)dxdy$$
 (see Lemma \ref{lem:dominant}), the sum in (\ref{sum}) equals asymptotically
	$$ \frac{1}{32 \pi}\phi (u)^2 u^2(u^2-1)^2 \int_{[-N,N]^2}\int_{[-N,N]^2}g_2^2(r)dxdy  .$$

We recall the definition of the $\upsilon$ function in  \eqref{g}. 
It follows that, 
 as $N\to \infty$, the variance of the modified EPC is asymptotically equivalent to  
 $$  \frac{1}{32 \pi}\upsilon(u) \int_{[-N,N]^2}\int_{[-N,N]^2}g_2^2(\|x-y\|)dxdy.$$
 The last integral is evaluated in Lemma \ref{lem:dominant} and it is equal to
 $$4   (2N)^3\,\frac{1}{\pi} \left(\frac{1-\sqrt{2}}{3} + \ln(1 + \sqrt{2})\right).$$ This concludes the proof of Proposition \ref{th:variance}.

\subsection{Proof of Proposition \ref{Corr}}\label{PreuveProp22}

In Proposition \ref{th:variance} we proved that the Euler-Poincaré characteristic behaves as the second chaotic projection, namely
$$\mathcal{\varphi}(f,T_N,u)-\mathbb{E}[\mathcal{\varphi}(f,T_N,u)]=\mathcal{\varphi}(f,T_N,u)[2] + o(\mathcal{\varphi}(f,T_N,u)[2]) .$$ Moreover, recalling that
$$\mathcal{D}_2 (f,N)=\int_{[-N,N]^2} \frac{1}{2}c_2 H_2(Y_2(x))+\frac{1}{2} c_5 H_2(Y_5(x))+\frac{1}{2} c_z H_2(Z(x))+ c_{z5}H_1(Z(x))H_1(Y_5(x))dx,$$
(see Equation \eqref{D2}) we also proved that
$$\varphi(f,T_N,u)[2]= \mathcal{D}_2(f,N)+O(N^2\sqrt{N}).$$

We now show that the entire Euler-Poincaré functional behaves as the second chaos projection of the initial field $f$ (as it happens for all the three Lipschitz-Killing curvatures in the sphere, see for example \cite[Equation (7)]{CM18}. We note that 
    $$H_2(f(x))=\frac{1}{3}H_2(Z(x))+ \frac{2}{3}H_2(Y_5(x))+\frac{2\sqrt{2}}{3}H_1(Z(x))H_1(Y_5(x)),$$
    since $f$ is a Laplace eigenfunction. We can now compute the necessary covariances and variances explicitly, term by term. We recall the notation $r:=\|x-y\|$.    Applying the Diagram Formula (see, e.g., \cite[Section 4.3.1]{MP11}), we have that
      \begin{eqnarray}\label{cov}
   &&\mathbb E \left[\mathcal{D}_2(f,N) \mathcal{S}(f,N)\right]  
          =\int_{[-N,N]^2}\int_{[-N,N]^2} \frac{1}{2}c_2 \frac{1}{3} 2\left(\frac{4}{\sqrt{3\cdot 8}}g_4(r)\right)^2 +2\frac{1}{2}c_2\frac{2}{3}\left(\frac{4}{\sqrt{3}}g_4(r)\right)^2\nonumber\cr
        && \qquad +2c_2 \frac{\sqrt{2}}{3} \left(-\frac{4}{\sqrt{3\cdot 8}}g_4(r) \cdot \frac{4}{\sqrt{3}}g_4(r)\right)+2\frac{1}{2}c_5 \frac{1}{3}\left(-\frac{\sqrt{8}}{3}g_7(r)\right)^2\nonumber\cr
        && \qquad  +2\frac{1}{2}c_5 \frac{2}{3} \left(\frac{8}{{3}}g_7(r)\right)^2 +2 c_5 \frac{\sqrt{2}}{3}\left(-\frac{\sqrt{8}}{3}g_7(r) \cdot \frac{8}{3}g_7(r)\right)\nonumber\cr
        && \qquad +2\frac{1}{2}c_z \frac{1}{3}\left(\frac{1}{{3}}g_7(r)\right)^2 +2\frac{1}{2}c_z \frac{2}{3}\left(-\frac{\sqrt{8}}{3}g_7(r) \right)^2 +2c_z \frac{\sqrt{2}}{3} \left(-\frac{\sqrt{8}}{3}g_7(r) \cdot \frac{1}{3}g_7(r)\right) \nonumber\cr
        && \qquad +2c_{z5}\frac{1}{3} \left(-\frac{\sqrt{8}}{3}g_7(r) \cdot \frac{1}{3}g_7(r)\right) +2c_{z5}\frac{2}{3} \left(-\frac{\sqrt{8}}{3}g_7(r) \cdot \frac{8}{3}g_7(r)\right) \nonumber\cr
        && \qquad + c_{z5} \frac{\sqrt{2}}{3} 2\left(\frac{8}{3}g_7(r)\frac{1}{3}g_7(r)+ \left(-\frac{\sqrt{8}}{3}g_7(r)\right)\left(-\frac{\sqrt{8}}{3}g_7(r)\right)\right).
    \end{eqnarray}
    Due to the equivalence
    $$\int_{[-N,N]^2}\int_{[-N,N]^2}g_2^2(r)dxdy \sim \int_{[-N,N]^2}\int_{[-N,N]^2}g_4^2(r)dxdy \sim \int_{[-N,N]^2}\int_{[-N,N]^2} g_7^2(r)dxdy$$  (see   Lemma \ref{lem:dominant}),     Equation 
 \eqref{cov} is equivalent to
    \begin{equation}\label{cov2}
        \frac{1}{4 \sqrt{\pi}} \phi (u) u (u^2-1) \int_{[-N,N]^2}\int_{[-N,N]^2} g_7^2(r)dxdy.
    \end{equation}
    Moreover, by direct calculation we obtain
\begin{equation}\label{varS}
        \mathbb E  \left[\left(\int_{[-N,N]^2}  H_2(f(x))dx\right)^2\right]= 2 \int_{[-N,N]^2}\int_{[-N,N]^2} g_7^2(r)dxdy.
    \end{equation}
    The value of this integral is given in Lemma \ref{lem:dominant}, postponed to Appendix \ref{sec:technical}.   Finally, by using  the variance of $\mathcal{D}_2(f,N)$  in Proposition \ref{th:variance},  the expression in \eqref{varS} and the covariance in \eqref{cov},  we get the thesis of Proposition \ref{PreuveProp22}.
  
\section{Proof of  Central Limit Theorem  \ref{TCLperturbed}}\label{sec:proof main}

{Firstly notice that the considered  Berry random field satisfies sufficient conditions in Remark 2 in \cite{DiBEO22}. Then one can use the Gaussian Kinematic Formula   (see Theorem 13.2.1  in \cite{AT07} or   Theorem 4.8.1 in \cite{AT11}) to get the expectation in Equation \eqref{meanECperturbed} (see also Equation (23) in  \cite{DiBEO22}).}\\

We  now state and prove the following result which provides the asymptotic variance of the modified EPC of the perturbed model in \eqref{Perturbation}.
\begin{proposition}\label{th:perturbed} 
By using Equation \eqref{PertSET}, we get   
{\begin{eqnarray*}
    (2N)^{-3}\Var(\varphi(f_\Lambda, T_N,u))&=& (2N)^{-3}
\Var(\varphi(f, T_N, h(u, \Lambda))\nonumber\\&=&  
\mathbb{ E}\left[\upsilon\left(h(u,\Lambda)\right)\right] \frac{1}{8\pi^2}  
\, \left(\frac{1-\sqrt{2}}{3} +\ln(1 + \sqrt{2})\right) +O\left(\frac{1}{\sqrt{N}}\right)  ,
\end{eqnarray*}}
with  $\upsilon$  as in Equation \eqref{g}.  
\end{proposition} 

\begin{proof}[Proof of Proposition \ref{th:perturbed}]
In Proposition \ref{th:variance} we proved that the second chaotic projection is the leading term of the series expansion in (\ref{Proj}) and then 
\begin{eqnarray*}
    \mathcal{\varphi}(f,T_N,u)-\mathbb{E}[\mathcal{\varphi}(f,T_N,u)]&=& \mathcal{\varphi}(f,T_N,u)[2]+ \sum_{\substack{q=1},q\ne 2}^\infty \sum_{\substack{\overline{n}\in\mathbb N^5,\\ |{n}|=q}} \langle G_u,\, \overline{H_{{n}}}\rangle h_{{n}}
\end{eqnarray*}    
where
$\mathcal{\varphi}(f,T_N,u)[2]=\mathcal{D}_2(f,N)+O(N^2\sqrt{N}).$ Therefore, by independence and monotone convergence theorem, conditioning with respect to $\Lambda$ and exploiting Proposition \ref{th:variance}, in particular re-adapting computations in (\ref{sum}) to the perturbed model, we obtain, as $N\to \infty$, 
\begin{eqnarray*}
&&\hspace{-0.2cm}\mathbb E \left[( \mathcal{\varphi}(f_\Lambda,T_N,u)-\mathbb{E}[\mathcal{\varphi}(f_\Lambda,T_N,u)])^2\right]  
= {\frac{1}{32\pi}}\mathbb E \left[ \upsilon\left(h(u,\Lambda)\right)\right] \mathbb E[F_N^2]+ \mathbb E \left[ \sum_{\substack{q=1,\\ q\neq 2}}^\infty \sum_{\substack{\overline{n}\in\mathbb N^5,\\ |{n}|=q}} \langle G_{ h(u,\Lambda)},\, \overline{H_{{n}}}\rangle^2 h_{{n}}^2\right]\\
   && = {\frac{1}{32\pi}} \mathbb E \left[ \upsilon\left(h(u,\Lambda\right)\right] \mathbb E[F_N ^2]+ \sum_{\substack{q=1,\\ q\neq 2}}^\infty \sum_{\substack{{n}\in\mathbb N^5,\\ |{n}|=q}} \mathbb E \left[ \langle G_{h(u,\Lambda)},\, \overline{H_{{n}}}\rangle^2 \right] \mathbb E \left[  h_{{n}}^2\right],
\end{eqnarray*}
with $\upsilon(u)$ defined in \eqref{g} and $$\mathbb{E}[F_N^2] = \int_{[-N,N]^2}\int_{[-N,N]^2}g_2^2(\|x-y\|)dxdy,$$ which has been computed in Lemma \ref{lem:dominant}. We write $\upsilon(u,q):= \mathbb E \left[ \langle G_{h(u,\Lambda)},\, \overline{H_{{n}}}\rangle^2 \right]$ and obtain that the variance equals
$${\frac{1}{32\pi}}\mathbb E \left[ \upsilon\left(h(u,\Lambda)\right)\right] \mathbb E[F_N^2] + \sum_{\substack{q=1,\\ q\neq 2}}^\infty \sum_{\substack{\overline{n}\in\mathbb N^5,\\ |\overline{n}|=q}} \upsilon(u,q) \mathbb E \left[  h_{{n}}^2\right].$$
We can note that for fixed $u$ the second summand is dominated by the first summand in $N$, and therefore the variance is of order of the first summand and this concludes the proof.
\end{proof}

Finally, let us prove the CLT for the Berry perturbed field given in Theorem \ref{TCLperturbed}.

\begin{proof}[Proof of Theorem \ref{TCLperturbed}]
    We show convergence of the distribution function. Let us denote by $D_\Lambda$ the image of $\Lambda$ and $F_\Lambda$ the distribution function of $\Lambda$. Let $x\in\mathbb R$, then we have by the law of total probability
    \begin{eqnarray*}
        \mathbb P && \biggl(\frac{\mathcal{\varphi}(f,T_N,h(u,{\Lambda}))- \mathbb{ E}[\mathcal{\varphi}(f,T_N,h(u,{\Lambda}))]}{\sqrt{Var(\mathcal{\varphi}(f,T_N,h(u,{\Lambda})))}}  \leq x\biggl)\\
        &&\qquad  =\int_{D_\Lambda} \mathbb P \left(\frac{\mathcal{\varphi}(f,T_N,h(u,{\Lambda}))-\mathbb{ E}[\mathcal{\varphi}(f,T_N,h(u,{\Lambda}))]}{\sqrt{Var(\mathcal{\varphi}(f,T_N,h(u,{\Lambda})))}} \leq x\:\Big| \: \Lambda =\lambda \right)  dF_\Lambda(\lambda ) \\
        &&\qquad =\int_{D_\Lambda} \mathbb P \left(\frac{\mathcal{\varphi}(f,T_N,h(u,{\lambda}))-\mathbb{ E}[\mathcal{\varphi}(f,T_N,h(u,{\lambda}))]}{\sqrt{Var(\mathcal{\varphi}(f,T_N,h(u,{\lambda})))}} \leq x\:\Big|\: \Lambda =\lambda \right) dF_\Lambda(\lambda ) \\
        &&\qquad =\int_{D_\Lambda} \mathbb P \left(\frac{\mathcal{\varphi}(f,T_N,h(u,{\lambda}))-\mathbb{ E}[\mathcal{\varphi}(f,T_N,h(u,{\lambda}))]}{\sqrt{Var(\mathcal{\varphi}(f,T_N,h(u,{\lambda})))}} \leq x \right)dF_\Lambda(\lambda ) ,
    \end{eqnarray*}
    since $f$ and $\Lambda$ are independent. Denoting by $Z$, as usual, a standard Gaussian random variable, by the dominated convergence theorem combined with the unperturbed CLT given in Theorem \ref{TCL}, this expression tends to
    $$  \int_{D_\Lambda} \mathbb P \left(Z \leq x \right) dF_\Lambda(\lambda ) = \mathbb P \left(Z \leq x \right) $$
    as $N$ tends to infinity, and the statement is proved.
\end{proof}
\section{Statistical applications}\label{NUmeric}

\subsection{Some parametric models for the perturbed Berry's random wave}\label{SpecialCases}
 
In this section, we consider some particular choices of the random shape variable   $\Lambda$ and a specific strictly increasing link function $g$ in \eqref{Perturbation} for which the variance in Theorem \ref{TCLperturbed}   can be explicitly computed. By Equation \eqref{variancePerturbed}, this means computing the term $\mathbb{ E}\left[\upsilon\left(h(u,\Lambda)\right)\right]$, with  function $\upsilon$ as in Equation \eqref{g} and $h(\cdot,\lambda)$ the inverse function   of $g(\cdot,\lambda)$ in the perturbed model  in  \eqref{Perturbation}. One can design several scenarios in which this term can be determined. In particular, if the perturbation is a scaling or a location shift, it suffices to know the value of $\mathbb E[\phi\left(h(u,\Lambda)\right)]$. In the following we provide an exemplary illustration for such a computation.  
 \begin{sloppypar} Suppose that the link function is a product, that is, $h(u,\Lambda)=\frac u {\Lambda}$. The authors in \cite{DiBEO22} computed explicitly $\mathbb{ E}\left[\rho_j\left(\frac u {\Lambda}\right)\right]$ for $j=0,1,2$,  where $\rho_j(x)=(2\pi)^{-(1+j)/2} e^{-x^2/2} H_{j-1}(x)$, for different choices of $\Lambda$. Using the formula for $j=2$, they could obtain an expression for the expectation of the EPC (see Equation \eqref{meanECperturbed}). We shall make use of these calculations to compute the asymptotic variance of the modified EPC. \end{sloppypar}  

First, we can combine the recursive formula
$$\rho_{j+1}(u)=-(2\pi)^{-1/2}\rho'_j(u)$$
(consequence of 22.8.8 in \cite{abramowitz+stegun}) and the relation
$$H_{n+1}(x)= xH_n(x)-nH_{n-1}(x)$$
(see 22.7.14 in \cite{abramowitz+stegun}) to obtain the expression
$$\mathbb E \left[\rho_j\left(\frac u {\Lambda}\right)\right] = -\frac{1}{2\pi} \left(u \frac{d}{du}\mathbb E \left[\rho_{j-2}\left(\frac u {\Lambda}\right)\right] + (j-2)\mathbb E \left[\rho_{j-2}\left(\frac u {\Lambda}\right)\right] \right).$$

Note also that, 
letting $\tilde{u} = \sqrt{2} \,u$, one can write 
\begin{eqnarray*}
    \mathbb{ E}\left[\upsilon\left(\frac u {\Lambda}\right)\right]   = \mathbb{ E}\left[ \phi \left(\frac u {\Lambda}\right)^2 \left(\frac {u^2} {\Lambda^2}-1\right)^2  \frac {u^2} {\Lambda^2}\right]  
    = \mathbb{ E}\left[ \frac{e^{-\frac{u^2}{\Lambda^2}}}{2\pi}   \left(\frac {u^2} {\Lambda^2}-1\right)^2  \frac {u^2} {\Lambda^2}\right]= \mathbb{ E}\left[ \frac{e^{-\frac{\tilde{u}^2}{2\Lambda^2}}}{2\pi}  \left(\frac {\tilde{u}^2} {2 \Lambda^2}-1\right)^2   \frac {\tilde{u}^2} {2\Lambda^2}\right].
\end{eqnarray*}  Expanding the polynomial part of the expression with Hermite polynomials yields
\begin{eqnarray}\label{Eupsilon}
    &&\mathbb{ E}\left[\upsilon\left(\frac u {\Lambda}\right)\right] = \nonumber\\
    &&\frac{1}{\sqrt{2\pi}} \bigg(\frac{1}{8}\mathbb E \left[ \phi \left(\frac {\tilde{u}} {\Lambda}\right) H_6 \left(\frac {\tilde{u}} {\Lambda}\right) \right]+ \frac{11}{8}\mathbb E \left[ \phi \left(\frac {\tilde{u}} {\Lambda}\right) H_4 \left(\frac {\tilde{u}} {\Lambda}\right) \right] + \frac{25}{8}\mathbb E \left[ \phi \left(\frac {\tilde{u}} {\Lambda}\right) H_2 \left(\frac {\tilde{u}} {\Lambda}\right) \right] \nonumber \\&&\quad + \frac{7}{8}\mathbb E \left[ \phi \left(\frac {\tilde{u}} {\Lambda}\right) H_0 \left(\frac {\tilde{u}} {\Lambda}\right) \right]\bigg) \nonumber \\
    && = \frac{1}{8}\left( (2\pi)^3 \mathbb E\left[\rho_7 \left(\frac {\tilde{u}} {\Lambda}\right)\right]+ 11 (2\pi)^2 \mathbb E\left[\rho_5 \left(\frac {\tilde{u}} {\Lambda}\right)\right] +25 (2\pi) \mathbb E\left[\rho_3 \left(\frac {\tilde{u}} {\Lambda}\right)\right]+ 7 \mathbb E\left[\rho_1 \left(\frac {\tilde{u}} {\Lambda}\right)\right]\right).
\end{eqnarray}

\textbf{Heavy-tail perturbation.}  Let $\Lambda$ such that $\Lambda^2$ is a random variable with a Pareto (Type I) distribution  with parameter $\alpha>0$ ($\Lambda^2\sim$ Pa($\alpha$)), i.e.,  
$\Pbb(\Lambda^2>x)= x^{-\alpha }, $ for $x> 1$  and $\alpha>0$.  Then
\begin{eqnarray*}
   \mathbb E \left[\rho_0\left(\frac u {\Lambda}\right)\right] &=&  u^{-2\alpha}\,2^{\alpha-1}(\pi)^{-1/2}\,\gamma(\alpha+1/2,u^2/2)+\overline\Phi(u),   \\
  \mathbb E \left[\rho_1\left(\frac u {\Lambda}\right)\right] &=& u^{-2\alpha}\,2^{\alpha-1} {(\pi)^{-1}}\,\alpha\,\gamma(\alpha,u^2/2),  \\
 \mathbb E \left[\rho_2\left(\frac u {\Lambda}\right)\right] &=&  u^{-2\alpha}\,2^{\alpha-1}{(\pi)^{-3/2}}\,\alpha\,\gamma(\alpha+1/2,u^2/2),
\end{eqnarray*}  
where $\gamma(a,\cdot)$ stands for the lower incomplete Gamma function, \emph{i.e.}, $\gamma(a,x)=\int_0^x t^{a-1}e^{-t}dt,$ for $a>0, x>0$ and   $\overline{\Phi}$ for the survival standard normal cumulative distribution function. Then we have   \begin{eqnarray*}
   \mathbb E \left[\rho_3\left(\frac u {\Lambda}\right)\right] &=& 
   -\frac{1}{(2\pi)}  2^{\alpha-1} {(\pi)^{-1}}\,\alpha (u^{-2\alpha}\gamma(\alpha,u^2/2)(1-2\alpha)+ 2^{1-\alpha}e^{-u^2/2}),  \\
  \mathbb E \left[\rho_5\left(\frac u {\Lambda}\right)\right] &=& 
\frac{1}{(2\pi)^2} 2^{\alpha-1} {(\pi)^{-1}}\,\alpha (u^{-2\alpha}\gamma(\alpha,u^2/2)(4\alpha^2-8\alpha +3)+ 2^{1-\alpha}e^{-u^2/2}(-2\alpha -u^2+4)), \\
   \mathbb E \left[\rho_7\left(\frac u {\Lambda}\right)\right] &=& 
 -\frac{1}{(2\pi)^3}2^{\alpha-1} {(\pi)^{-1}}\,\alpha (u^{-2\alpha}\gamma(\alpha,u^2/2)(-8\alpha^3+36\alpha^2-46\alpha+15)\\
 &&\qquad+ 2^{1-\alpha}e^{-u^2/2}(4\alpha^2+2\alpha u^2-18\alpha+u^4-11u^2+23)),  
\end{eqnarray*}
and then from \eqref{Eupsilon}
\begin{eqnarray}\label{variancePARETO}
 \mathbb{ E}\left[\upsilon\left(\frac u {\Lambda}\right)\right] &=& \frac{1}{8} 2^{\alpha-1} {(\pi)^{-1}}\,\alpha \big((\sqrt{2}u)^{-2\alpha}\gamma(\alpha,u^2)(8\alpha^3+8\alpha^2+8\alpha)-  \nonumber \\ && \quad 2^{1-\alpha}e^{-u^2}((\sqrt{2}u)^4+2\alpha (\sqrt{2}u)^2+4\alpha^2+4\alpha +4)\big). 
\end{eqnarray}

\textbf{Light-tail perturbation.}  
Let us consider $\Lambda$ such that $\Lambda^2 \sim \text{Exp}(\theta)$. In \cite{DiBEO22} the authors show that
$$\mathbb E \left[\rho_0\left(\frac u {\Lambda}\right)\right]= \frac{1}{2}e^{-a u},  \quad \mathbb E \left[\rho_1\left(\frac u {\Lambda}\right)\right] =  \frac{1}{2\pi} a u K_1(a u),  \quad     \mathbb E \left[\rho_2\left(\frac u {\Lambda}\right)\right]= \frac{1}{4 \pi}a u e^{-a u},$$
with $a= \sqrt{2/\theta}$.  Then, we have   \begin{eqnarray*}
 \mathbb E \left[\rho_3\left(\frac u {\Lambda}\right)\right] &=& 
   -\frac{1}{(2\pi)^2}   (-a^2u^2 K_0(au)+au K_1(au)),\\
  \mathbb E \left[\rho_5\left(\frac u {\Lambda}\right)\right] &=& 
\frac{1}{(2\pi)^3} (-6a^2u^2 K_0(au)+ K_1(au)(a^3u^3+3au)),\\
   \mathbb E \left[\rho_7\left(\frac u {\Lambda}\right)\right] &=& 
 -\frac{1}{(2\pi)^4} (K_1(au)(13a^3u^3+15au)-(45a^2u^2+a^4u^4) K_0(au)),
\end{eqnarray*}
where $K_\nu$ is modified Bessel functions of the second kind of order $\nu$. From Equation \eqref{Eupsilon},  we get 
 $$\mathbb{ E}\left[\upsilon\left(\frac u {\Lambda}\right)\right] = \frac{1}{8(2\pi)} \left(a^4\tilde{u}^4 K_0(a\tilde{u})-2 a^3\tilde{u}^3 K_1(a\tilde{u})+4a^2\tilde{u}^2 K_0(a\tilde{u})\right).$$
 
\subsection{Numerical studies} \label{numeric}
In this section we numerically illustrate the performance on finite  size  samples of some obtained  results  as the asymptotic variances and the asymptotic Gaussian fluctuations. In Figure  \ref{fig:BerryGeneration0} we generate via the   \texttt{R} package \texttt{RandomFields} (see \cite{Schlather2015}), the  Gaussian Berry random field $f$ as in Equation \eqref{covariance} (first panel)  in $T= [-N, N]^2$, with $N=70$, for different choices of the pixelization resolution.    In Figure \ref{fig:BerryGeneration},   we display the excursion sets for levels $u= 0, 0.5, 1$ and $2$ in  $[-70, 70]^2$, with   $140$ pixels for side.  

\begin{figure}[ht!]
\hspace{-0.4cm}
\includegraphics[width=4cm, height=4.1cm]{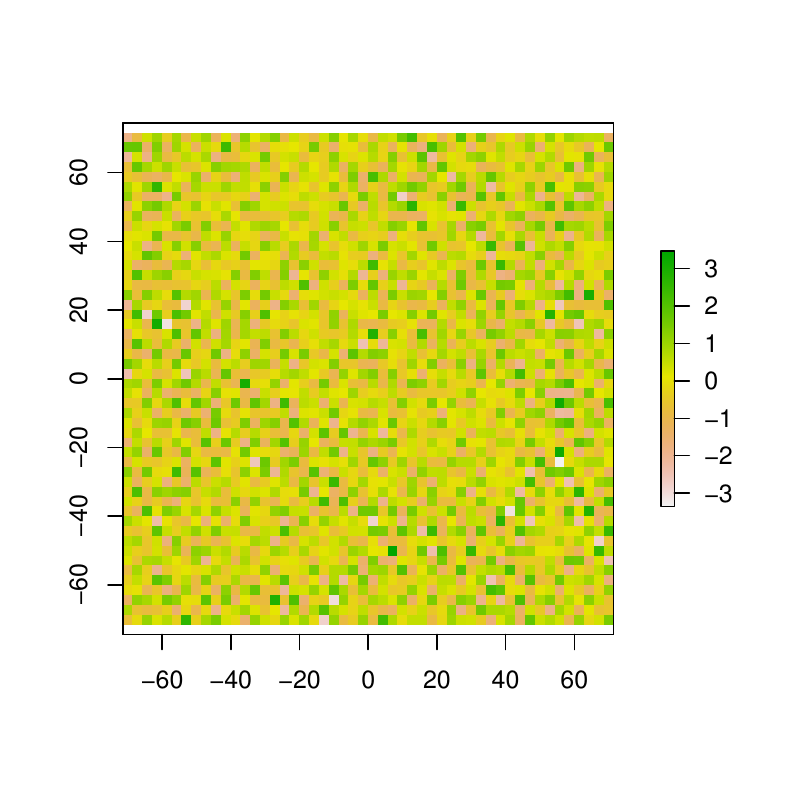}\hspace{-0.5cm}
 \includegraphics[width=4cm, height=4.1cm]{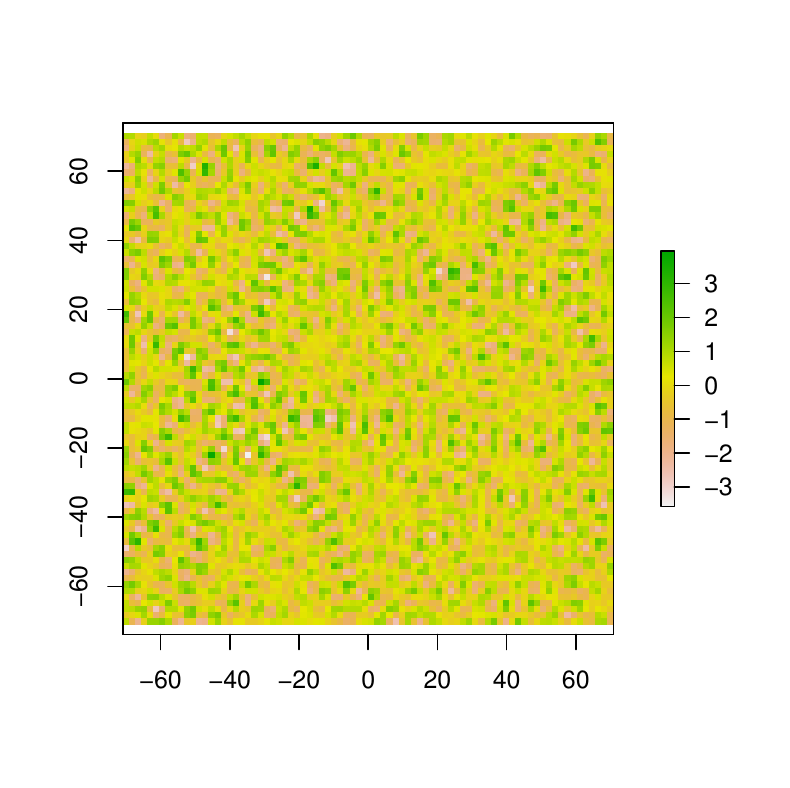}\hspace{-0.5cm}
 \includegraphics[width=4cm,height=4.1cm]{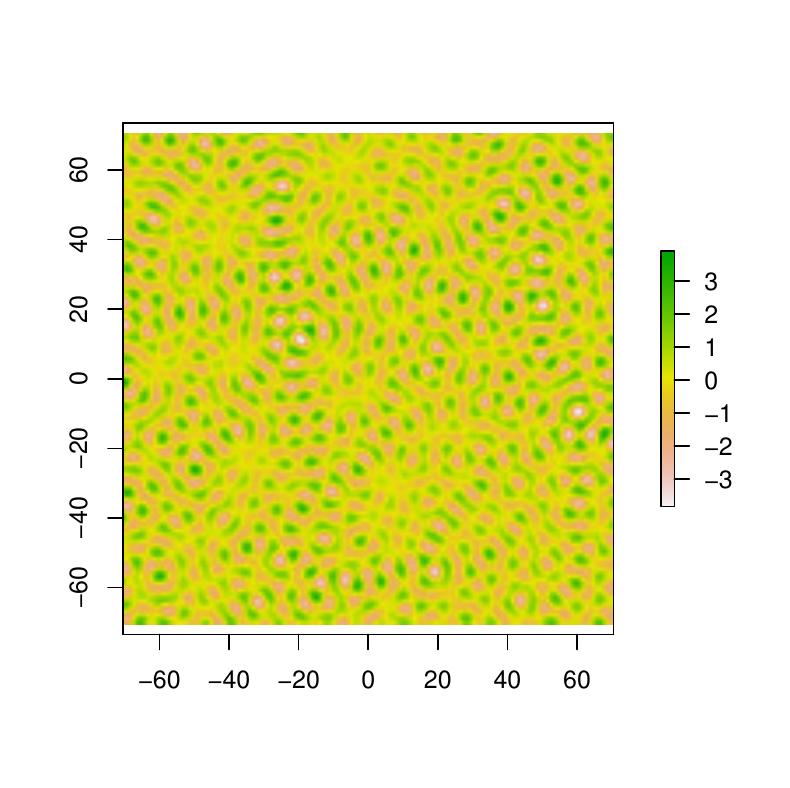}\hspace{-0.6cm}
 \includegraphics[width=4cm,height=4.1cm]{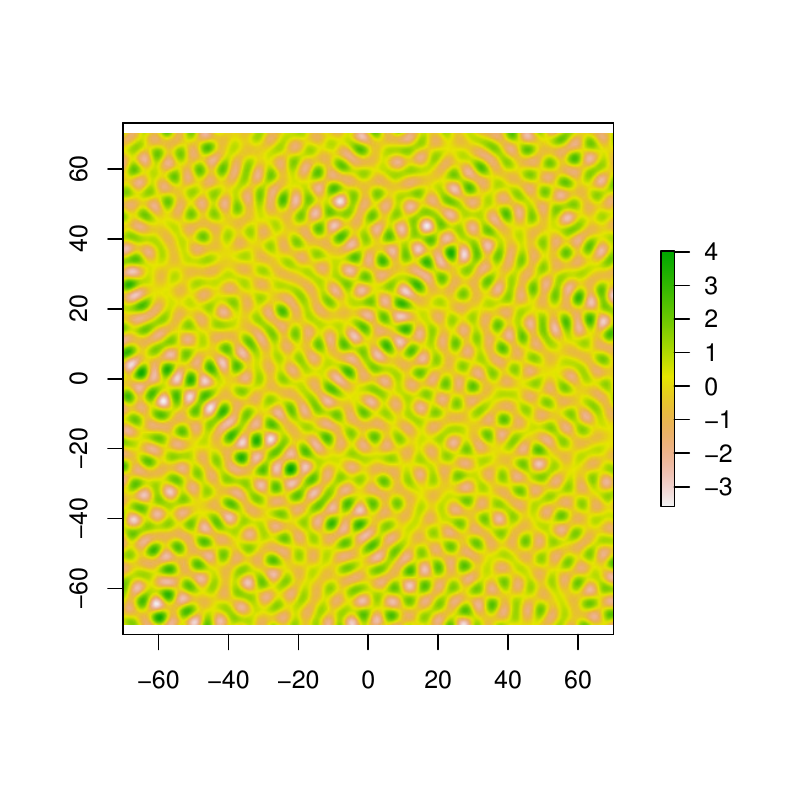} 
\caption{A random generation in $[-N, N]^2$, with $N=70$ of a Gaussian Berry random field $f$. From the left to the right: the side of the domain is divided in $50$, $80$, $140$  and $400$ pixels respectively.  Simulation are provided using the \texttt{R} package \texttt{RandomFields} (see \cite{Schlather2015}).\vspace{-0.3cm}}
\label{fig:BerryGeneration0}
\end{figure} 

\begin{figure}[ht!]
\hspace{-0.2cm}
\includegraphics[width=3.12cm,height=3.4cm]{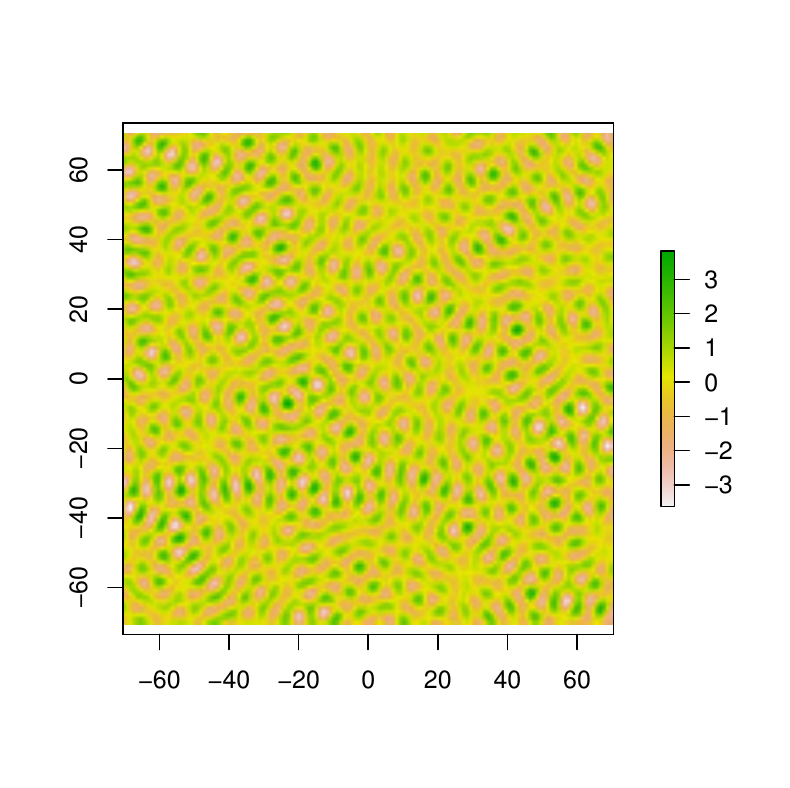}\hspace{-0.43cm}
\includegraphics[width=3.12cm,height=3.4cm]{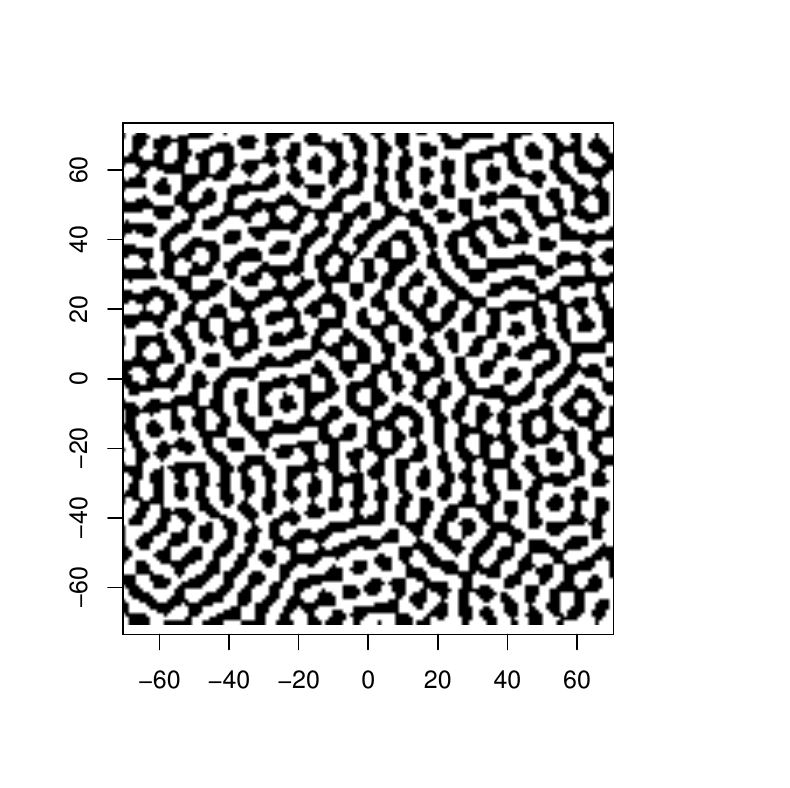}\hspace{-0.43cm}
\includegraphics[width=3.12cm,height=3.4cm]{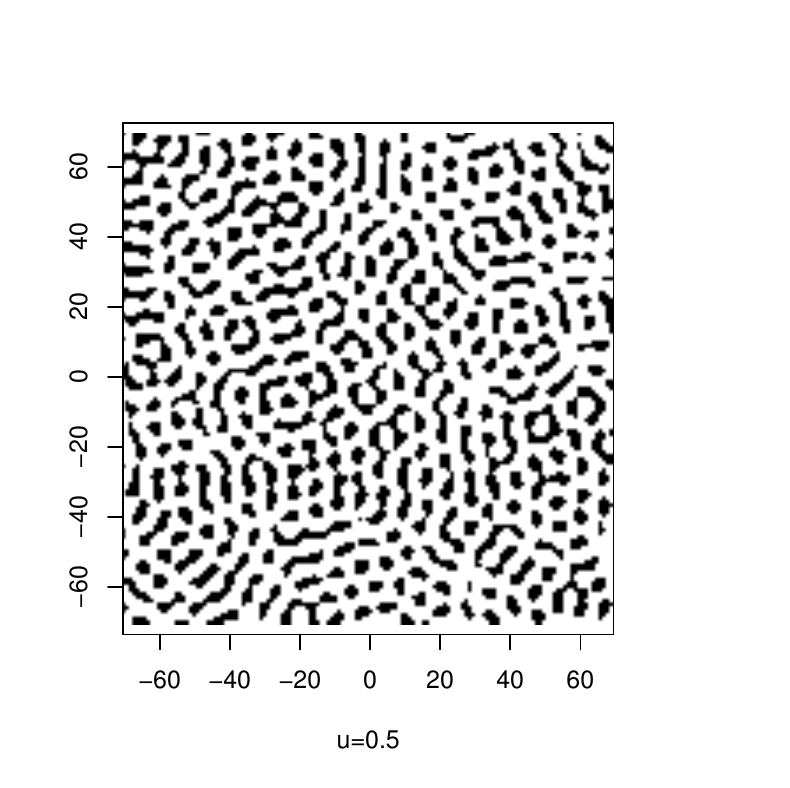}\hspace{-0.43cm}
\includegraphics[width=3.12cm,height=3.4cm]{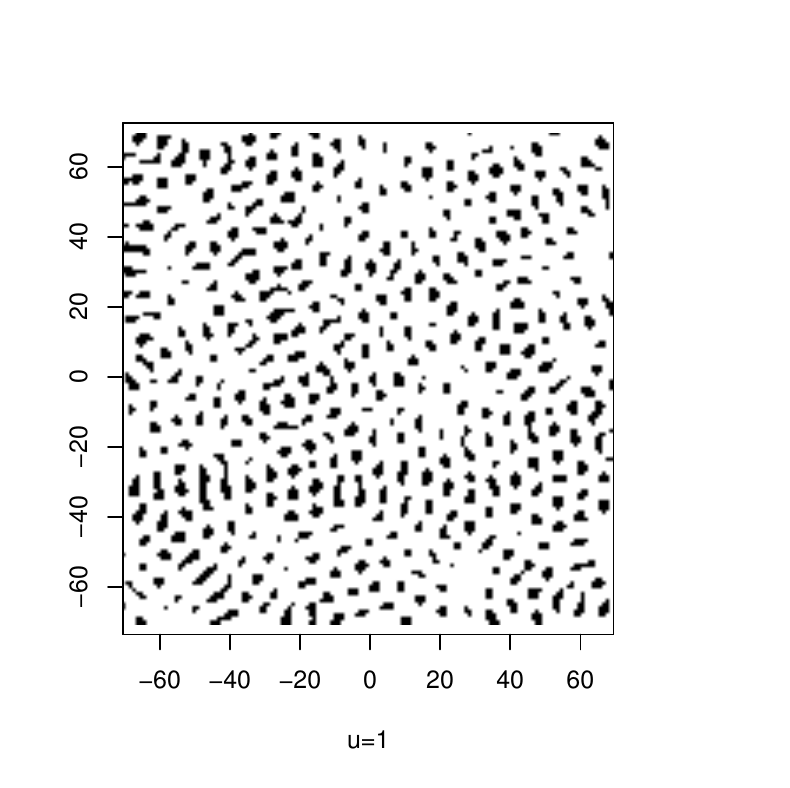}\hspace{-0.43cm}
\includegraphics[width=3.12cm,height=3.4cm]{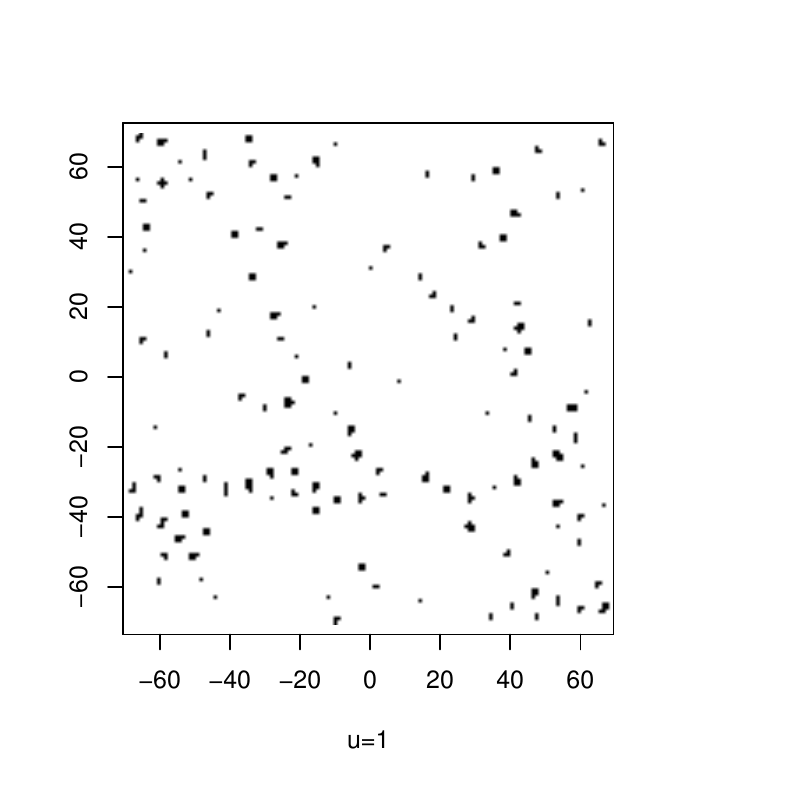}
\caption{A random generation in $[-70, 70]^2$, with $140$ pixels for side, of a Gaussian Berry random field $f$ (first panel) with associated excursion sets for  $u=0, 0.5, 1, 2$ (from the left to the right panels). Simulation are provided using the \texttt{R} package \texttt{RandomFields} (see \cite{Schlather2015}).}\label{fig:BerryGeneration}
\end{figure}

In Figure  \ref{fig:pvalues},   we illustrate the asymptotic Gaussian fluctuations stated  in Theorem \ref{TCL} for  growing domain (i.e., $N \to \infty$). To this aim, we display the $p$-values of the  Shapiro-Wilk Normality Test  for the statistics $(\varphi(f,T_N,u)-\mathbb{E}[\varphi(f,T_N,u)])/\sqrt{(2N)^3 V(u)}$, with $V(u)$  prescribed by Equation \eqref{eqVu}, as a function of levels $u$ and for several choices of $N$: $N= 70$ (left panel), $N=100$ (center panel), $N=200$ (right panel). 
The theoretical expected $\mathbb{E}[\varphi(f,T_N,u)]$ is given in \eqref{meanEC} and the estimation of the EPC $\varphi$ is evaluated on Berry's Gaussian samples in  $T$ and $2\,  N$ pixels for sides, using the \texttt{R} package \texttt{RandomFields} (see \cite{Schlather2015}). Notice the increasing performance of the $p-$values in terms of $N$ (from left to right panels in Figure \ref{fig:pvalues}).
\begin{figure}[ht!]  
\hspace{-0.1cm}\includegraphics[width=4.5cm, height=4.5cm]{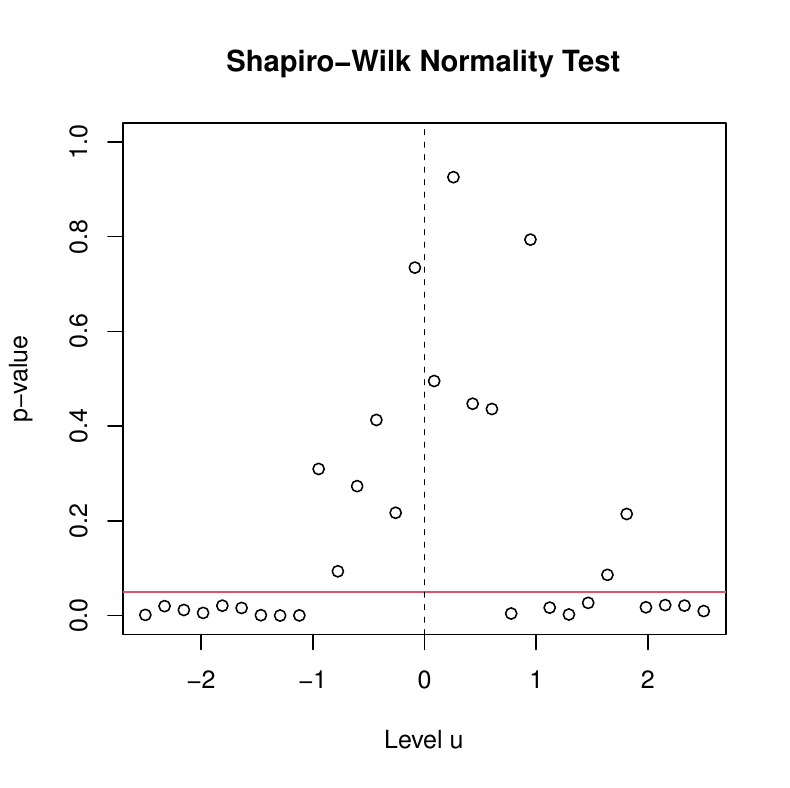} 
\hspace{-0.2cm}\includegraphics[width=4.5cm, height=4.5cm]{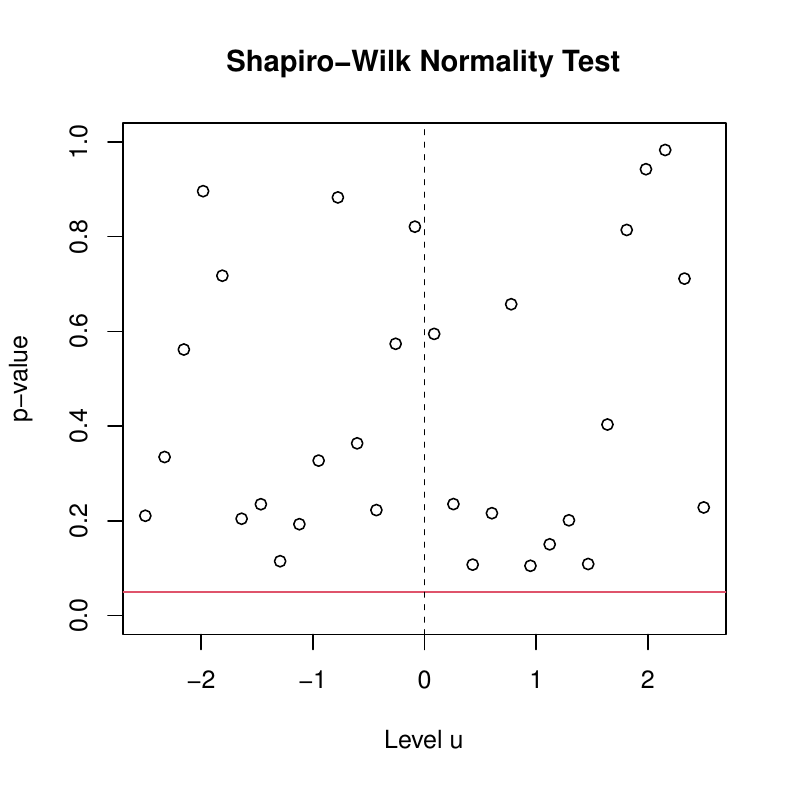}  \hspace{-0.2cm}\includegraphics[width=4.5cm, height=4.5cm]{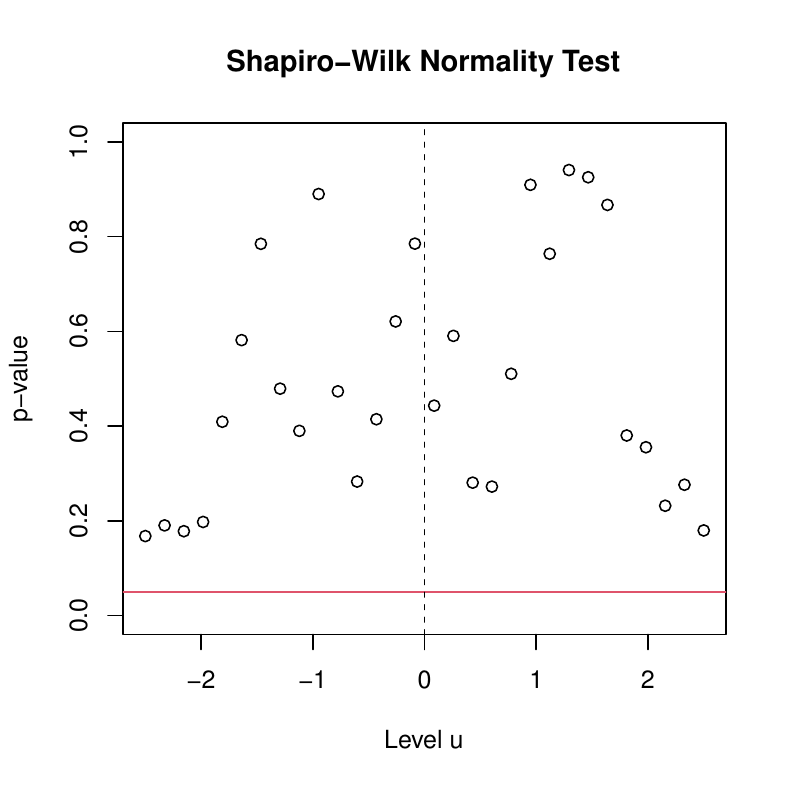} 
\caption{We display the $p$-values of the  Shapiro-Wilk Normality Test  for the statistics $(\varphi(f,T_N,u)-\mathbb{E}[\varphi(f,T_N,u)])/\sqrt{(2N)^3 V(u)}$  as a function of levels $u$ and  for several choices of the window  $T_N=[-N, N]^2$; $N= 70$ (left panel), $N=100$ (center panel), $N=200$ (right panel). 
  Estimation of the EPC $\varphi$ is provided on $250$ Berry's Gaussian samples  using the \texttt{R} package \texttt{RandomFields} (see \cite{Schlather2015}).\vspace{-0.3cm}}
\label{fig:pvalues}
\end{figure}
{Finally, we consider  the  Berry mixture model $f_\Lambda$ as in Equation \eqref{Perturbation}.    In Figure \ref{plotPerturbation} we display a realisation of the considered $f_\Lambda$ random field with $h(u, \Lambda)=\frac{u}{\Lambda}$, where  $\Lambda^2\sim$  Pa($\alpha$) with $\alpha = 4$ 
 and associated excursions set. Notice that the choice of this  Pareto  mixture clearly impacts on the standard deviation (and obviously on the range of values, see Figure \ref{plotPerturbation}) of the perturbed random field. 
 
In Figure \ref{plotPerturbationMEAN} (fist and second panels)  we display the  boxplots of the estimated  EPC values on  $500$ sample simulations for  $\alpha=2$    (first panel),  $\alpha=4$    (second panel). Unsurprisingly, the heavier-tail model for $\alpha=2$  exhibits an higher variance in the estimation (first panel in Figure \ref{plotPerturbationMEAN}). The estimated EPC is empirically evaluated on mixture Berry samples in  $[-N, N]^2$ with $N=70$ and $2N$ pixels for sides. 
Furthermore,  we add the theoretical expectation  $u \to \mathbb{E}[\varphi(A(f_\Lambda,T_N,u))]$, for $u \in [-3.5, 3.5]$ (red curve), by using Equation  \eqref{meanECperturbed} and the $\mathbb{E}[\rho_j(\frac{u}{\Lambda}))]$'s expressions obtained in Section \emph{Heavy-tail perturbation} before.} Finally the theoretical variance $u \to  V_\Lambda (u)$, for $u \in [-3.5, 3.5]$ 
 obtained by \eqref{variancePerturbed} and \eqref{variancePARETO} is displayed in the third panel of  Figure \ref{plotPerturbationMEAN} by using a red curve. Furthermore, the empirical variances obtained on 1000 Monte Carlo simulations are displayed in black dots. 
 
\begin{figure}[ht!]
\hspace{-0.2cm}
\includegraphics[width=3.12cm,height=3.4cm]{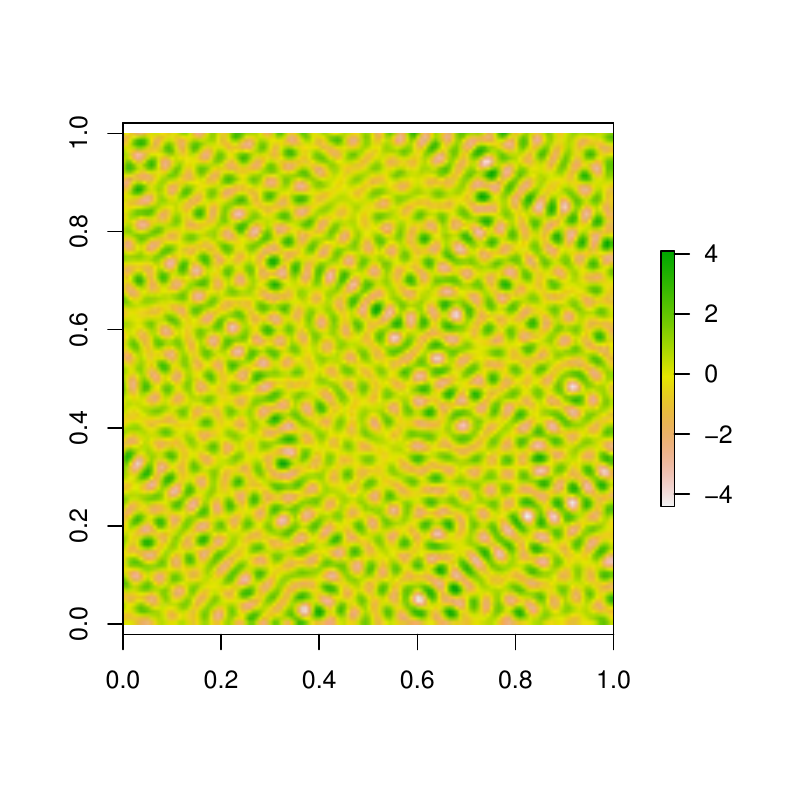}\hspace{-0.43cm}
\includegraphics[width=3.12cm,height=3.4cm]{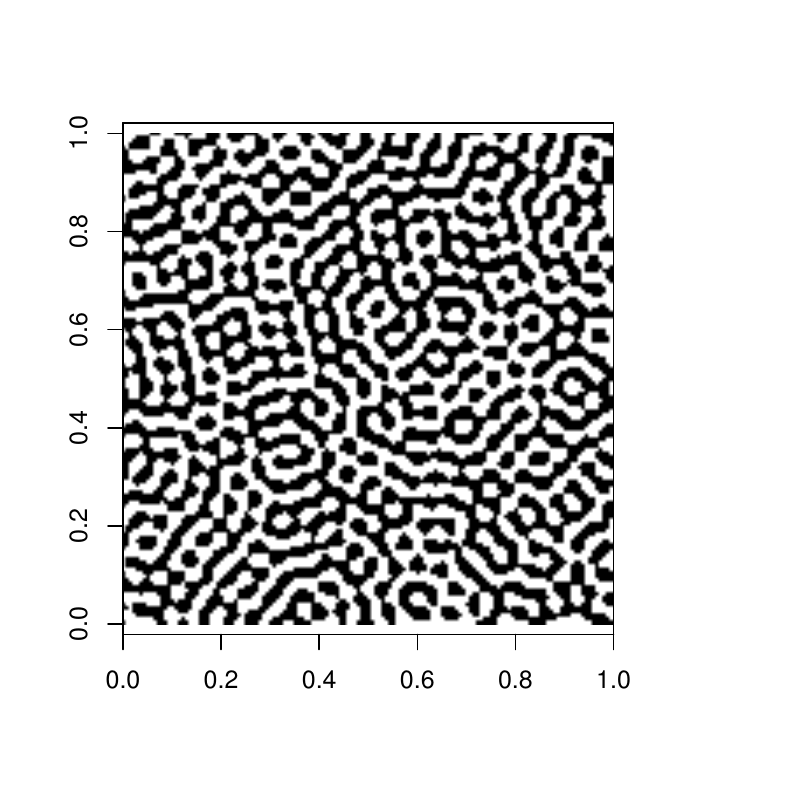}\hspace{-0.43cm}
\includegraphics[width=3.12cm,height=3.4cm]{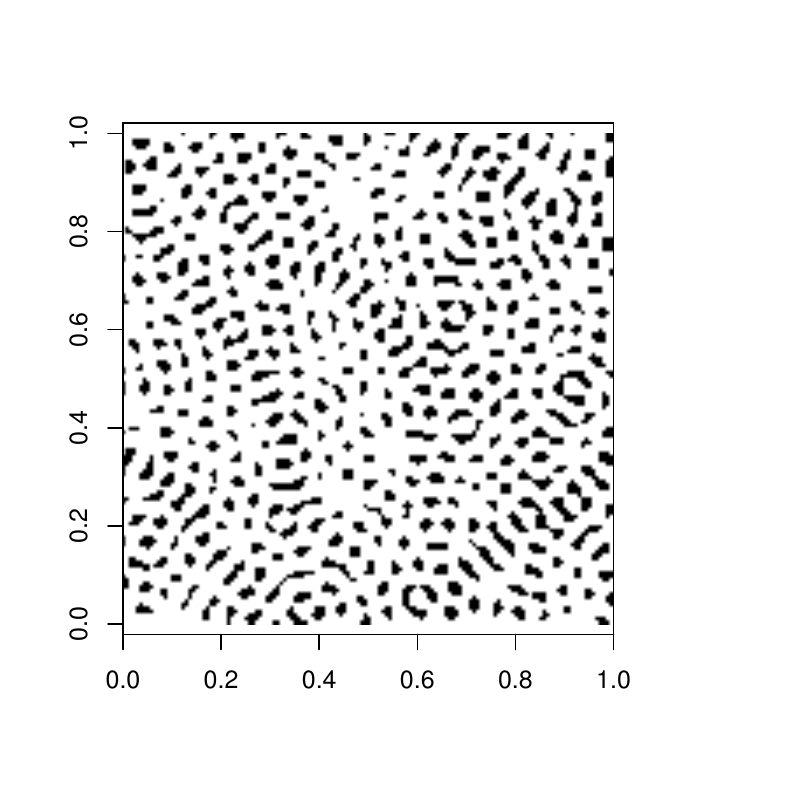}\hspace{-0.43cm}
\includegraphics[width=3.12cm,height=3.4cm]{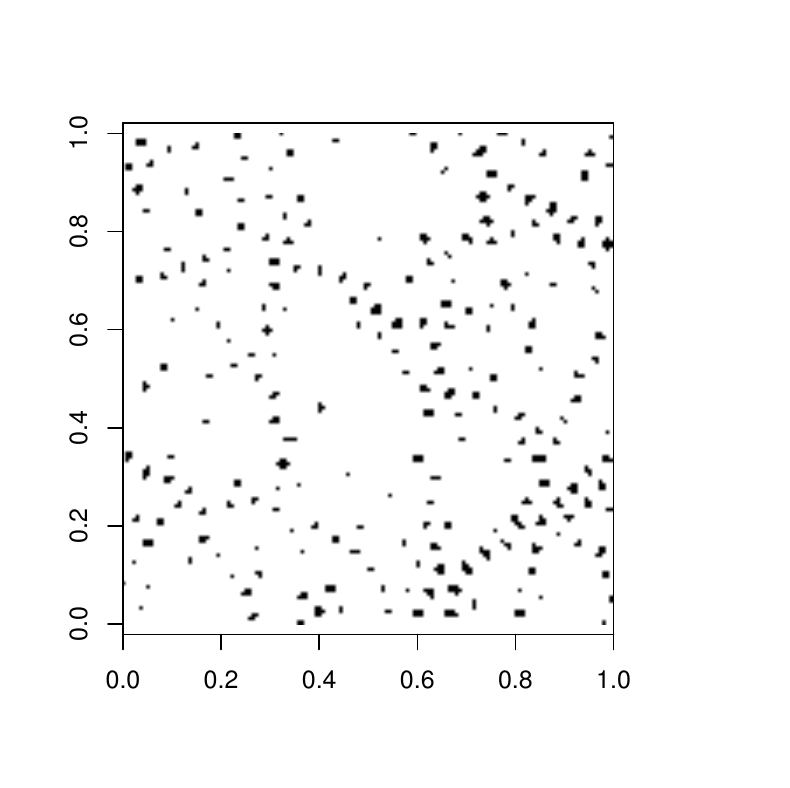}\hspace{-0.43cm}
\includegraphics[width=3.12cm,height=3.4cm]{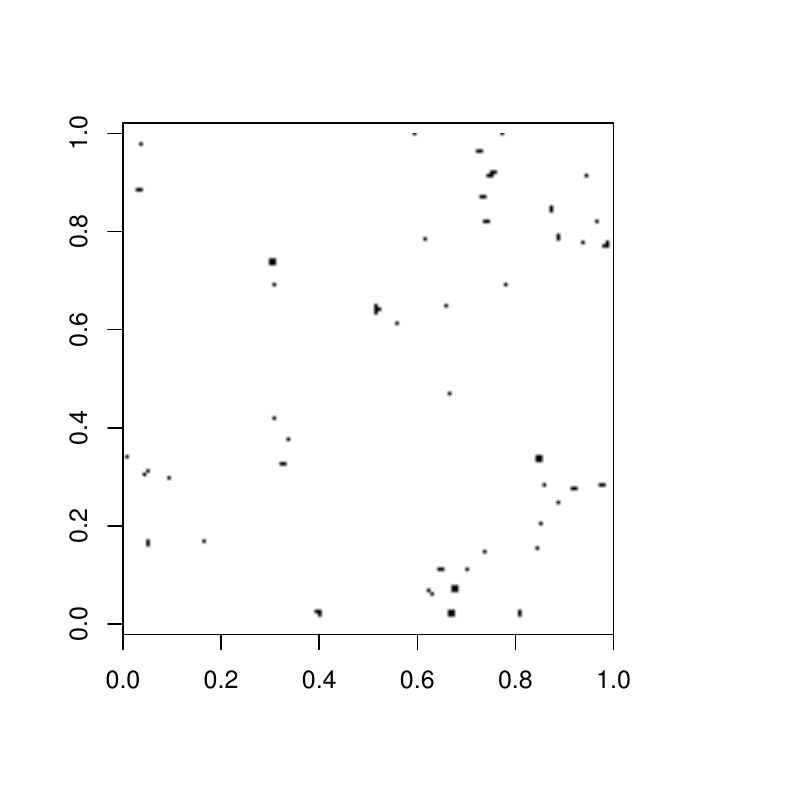}
\caption{A random generation in $[-N, N]^2$  of a mixture   Berry random field  $f_\Lambda$ as in Equation \eqref{Perturbation} with $h(u, \Lambda)=\frac{u}{\Lambda}$, where  $\Lambda^2\sim$   Pa($\alpha$) and $\alpha=4$    (first panel) with associated excursion sets for  $u=0, 1, 2, 3$ (from the second to the fifth panel). Here we take  $N=70$ and $2\,N$ pixels for side. Simulation are provided using the \texttt{R} package \texttt{RandomFields} (see \cite{Schlather2015}).}\label{plotPerturbation}\vspace{-0.2cm}
\end{figure}

\begin{figure}[ht!]
\includegraphics[width=4cm, height=4cm]{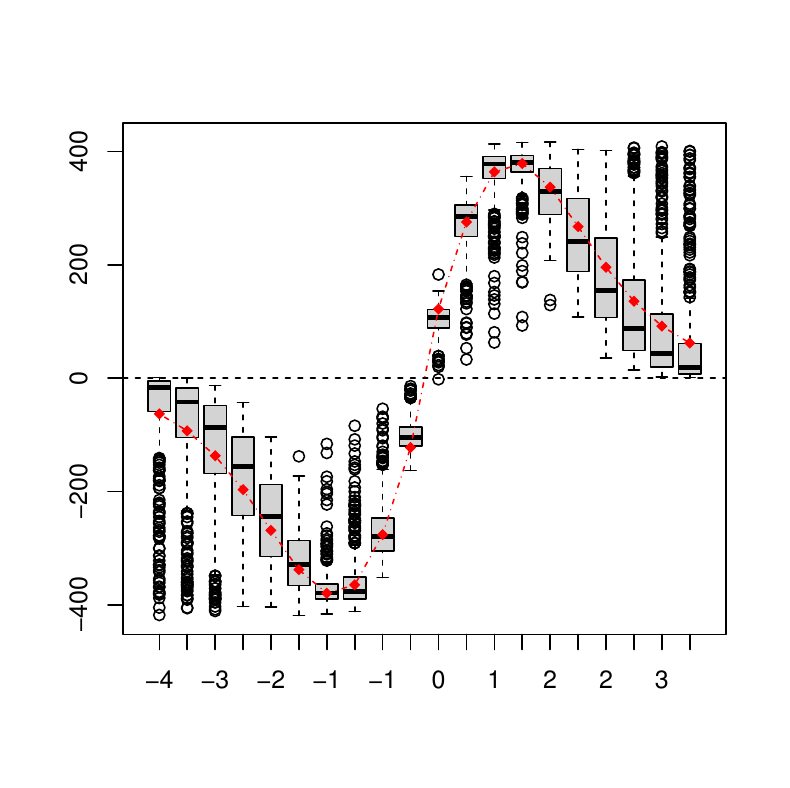}
\includegraphics[width=4cm, height=4cm]{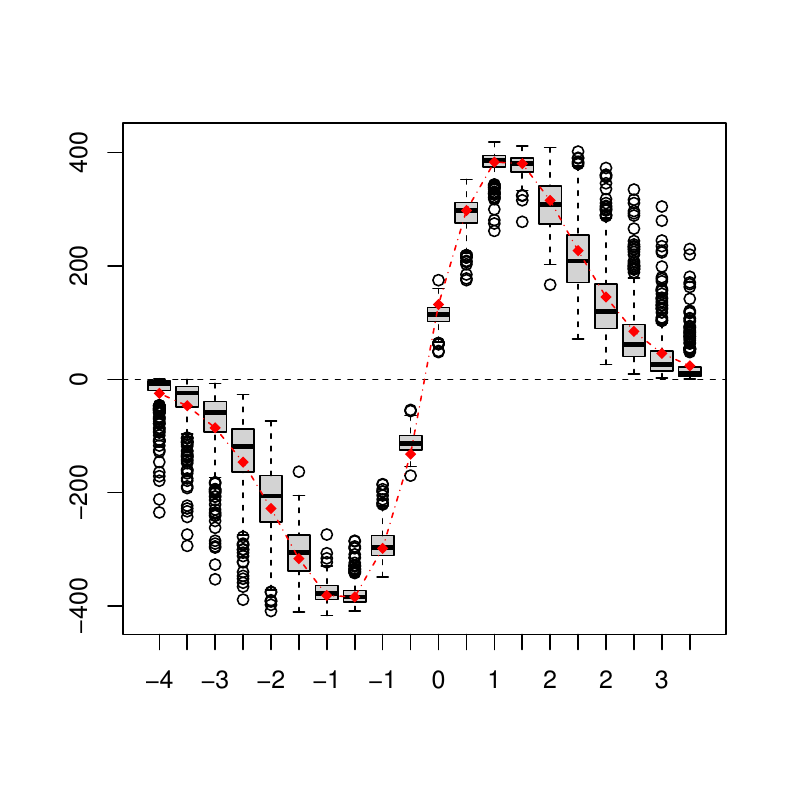}
\includegraphics[width=4cm, height=4cm]{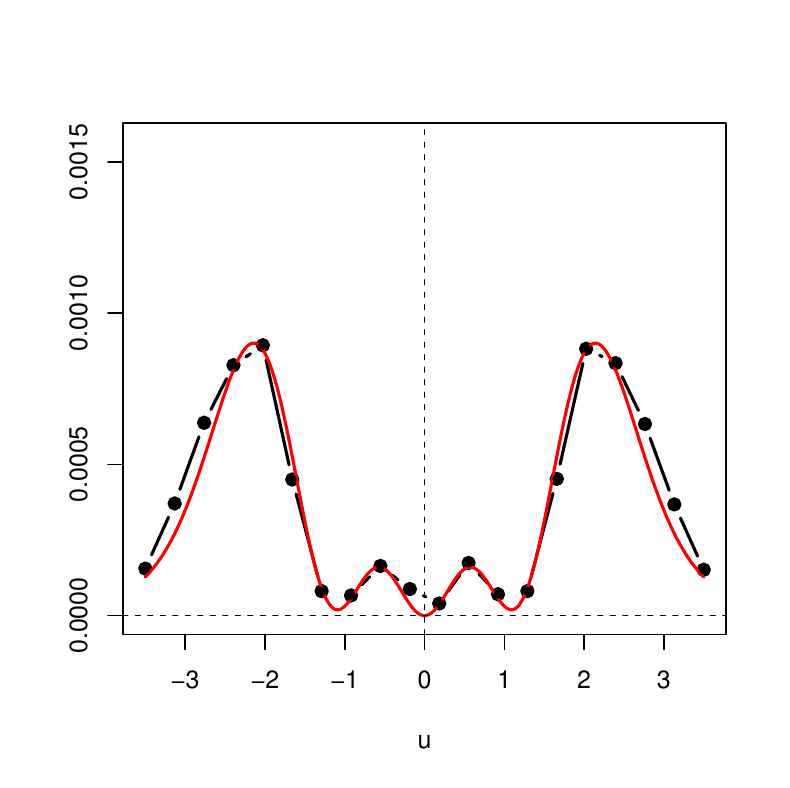}
\caption{We consider the perturbed model $f_\Lambda$ as in Equation \eqref{Perturbation},  with $h(u, \Lambda)=\frac{u}{\Lambda}$, where  $\Lambda^2\sim$ Pa($\alpha$)   $\alpha=2$    (first panel),  $\alpha=4$    (second and third panels). We display the averaged values for the EPC estimation on  $500$ sample simulations (black dots). Theoretical $u \to \mathbb{E}[\varphi(A(f_\Lambda,T_N,u))]$ (first and second panels) and $u \to  V_\Lambda (u)$ (third panel), for $u \in [-3.5, 3.5]$ 
are drawn in red line. Empirical counterpart are represent by using black boxplots and black dots.  Here we take   $[-N, N]^2$ with $N=70$ and $2\,N$ pixels for side. Simulation are provided using the \texttt{R} package \texttt{RandomFields} (see \cite{Schlather2015}).}\label{plotPerturbationMEAN}
\end{figure}

\appendix

\section{Technical lemmas} \label{sec:technical}
In this section we collect all the technical results exploited in the proof of Proposition \ref{th:variance} which give the asymptotic behaviour of the chaos components of the modified EPC expansion. In particular, we first state 
Lemma \ref{radial} which is an auxiliary result used in the proof of these technical lemmas below. Then, Lemma
 \ref{leading}, \ref{lem:5}, \ref{lem:oN3} and Remark \ref{1st-tech} imply that  the asymptotic behaviour of all the terms in the $q$-chaos for $q\ne 2$ is $o(N^3)$, as $N\to \infty$. Lemma \ref{lem:dominant} shows that the second chaos behaves as $N^3$ and hence it is the leading term of the chaos expansion. Finally in Lemma \ref{coefficients} we compute the constants of the dominant terms of the second chaotic projection.
 
To prove some of the technical lemmas we first need the following simple result.
\begin{lemma}\label{radial}
	Let $h:\mathbb \R \to \mathbb \R^+$ be a non-negative function. Then, assuming the integrals are well-defined, 
	$$\int_{[-N, N]^2}\int_{[-N, N]^2} h(\|x-y\|)dx\, dy \leq 8N^2 \pi \int_0^{2\sqrt 2 N}h(r) r dr.$$
\end{lemma}
\begin{proof}[Proof of Lemma  \ref{radial}]
	By a change of variables, we get
	\begin{eqnarray*}
		\int_{[-N, N]^2}\int_{[-N, N]^2} h(\|x-y\|)dxdy = \int_{-2N}^{2N}\int_{-2N}^{2N} (2N-|x_1|)(2N-|x_2|)h\left(\left\| \begin{pmatrix}
			x_1 \\ 
			x_2
		\end{pmatrix}\right\|\right)dx_1 dx_2,
	\end{eqnarray*}
	which can then be bounded by
	\begin{eqnarray*}
		4N^2 \int_{-2N}^{2N}\int_{-2N}^{2N} h\left(\left\| \begin{pmatrix}
			x_1 \\ 
			x_2
		\end{pmatrix}\right\|\right)dx_1 dx_2 
		 \leq 4 N^2\int_{B_{2\sqrt{2}N}(0)} \hspace{-0.25cm} h\left(\left\| \begin{pmatrix}
			x_1 \\ 
			x_2
		\end{pmatrix}\right\|\right)dx_1 dx_2 = 8N^2 \pi \int_0^{2\sqrt 2 N}\hspace{-0.25cm}h(r) r dr,
	\end{eqnarray*}
	with a switch to radial coordinates.
\end{proof}

The following lemma provides some auxiliary statements for the bounds of the chaos components.
\begin{lemma}\label{leading}
	We have for $\alpha \in \mathbb N^7, |\alpha|=q$, that, as $N\to\infty$,
	$$\int_0^{2\sqrt{2}N} r \prod_{i=1}^7 \left| g_i(r)\right|^{\alpha _i} dr = o(N)$$
	for $q\geq 3$ and for $q\geq 1$ if at least one of $\alpha_1$, $\alpha_3$, $\alpha_5$, $\alpha_6$ is nonzero. Moreover,
	$$\left|\int_0^{2N\sqrt{2}} r g_i(r) dr\right| = o(N)$$
	for $q=1$ in either case.
\end{lemma}

\begin{proof}[Proof of Lemma \ref{leading}]
	First recall the following asymptotic formula for $J_m$, given, for instance, in 9.2.1 in \cite{abramowitz+stegun}:
	\begin{equation}\label{Bessel}
	J_m(r)=\sqrt{\frac{2}{\pi r}}\left(\cos \left(r-\frac{m\pi}{2}-\frac{\pi}{4}\right)+O\left(\frac{1}{r}\right)\right),
	\end{equation}
	as well as the fact that Bessel functions are uniformly bounded by $1$.\\
	For $q=1$, the statement in the lemma is obvious for $g_1$, $g_3$, $g_5$ and $g_6$. More precisely, we have for some constant $K$ and $i=1,\,3,\,5$, or $6$
	$$\left|\int_0^{2\sqrt{2}N} r g_i(r) dr\right|\lesssim K + \int_1^{2\sqrt{2}N} \frac{1}{\sqrt r} dr \lesssim \sqrt{N} = o(N)$$
	thanks to the asymptotics and the bound mentioned above.\\
	The proof for $q\geq 3$ follows along the same lines.\\
	
	For $q=1$ and $g_2$, $g_4$, and $g_7$ we can compute the integral explicitly:
	\begin{eqnarray*}
		&&\int_0^{2\sqrt{2}N} r g_2(r) dr = J_0(2\sqrt{2}N)+2\sqrt{2}N J_1(2\sqrt{2}N)-1,\\
		&&\int_0^{2\sqrt{2}N} r g_4(r) dr = -2\sqrt{2}NJ_2(2\sqrt{2}N),\\
		&&\int_0^{2\sqrt{2}N} r g_7(r) dr = 2\sqrt{2}NJ_1(2\sqrt{2}N)-\frac{J_1(2\sqrt{2}N)}{2\sqrt{2}N}+J_0(2\sqrt{2}N)-J_2(2\sqrt{2}N)-\frac{1}{2}.
	\end{eqnarray*}
	Again, due to the asymptotics of the Bessel functions, these integrals are sublinear in $N$.
\end{proof}

\begin{remark}\label{1st-tech}
	We also have
	\begin{eqnarray*}
		&&\left|\int_0^{2\sqrt{2}N} r^2 g_2(r) dr\right| \lesssim N\sqrt{N}, \quad 
 \left|\int_0^{2\sqrt{2}N} r^2 g_4(r) dr\right|  \lesssim N\sqrt{N}, \quad \left|\int_0^{2\sqrt{2}N} r^2 g_7(r) dr\right|  \lesssim N\sqrt{N}
	\end{eqnarray*}
	as well as
	\begin{eqnarray*}
&&\left|\int_0^{2\sqrt{2}N} r^3 g_2(r) dr\right| \lesssim N^2\sqrt{N}, \quad \left|\int_0^{2\sqrt{2}N} r^3 g_4(r) dr\right|  \lesssim N^2\sqrt{N}, \quad \left|\int_0^{2\sqrt{2}N} r^3 g_7(r) dr\right|  \lesssim N^2\sqrt{N}. 
	\end{eqnarray*} 
	by direct calculation, using the identity 6.561 in \cite{gradshteyn2007} together with the asymptotics of the Bessel and Lommel functions (Equation \eqref{Bessel} and an approximation result from \cite{Dingle} respectively).
\end{remark}
 
{The following three lemmata (see Lemmas \ref{lem:5}, \ref{lem:oN3} and \ref{lem:dominant} below) provide bounds and/or precise calculations of all chaos components of the modified Euler-Poincaré characteristic.}
\begin{lemma}\label{lem:5}
	The variance of chaos components of order $3$ or higher, as well as all summands in the chaos decomposition containing at least one of the factors $\partial_1 f$, $\partial_{11} f$, and $\partial_{12} f$ is of order $o(N^3)$.
\end{lemma}
\begin{proof}[Proof of Lemma \ref{lem:5}]
	This is a direct consequence of Lemma \ref{radial} and Lemma \ref{leading}. 
\end{proof}
\begin{lemma}\label{lem:oN3}
	The variance of the first chaos component involving the factors $\partial_2 f$ and/or $\partial_{22} f$ is of order $o(N^3)$, that is,
	$$\left| \int_{[-N, N]^2}\int_{[-N, N]^2} g_i(\|x-y\|)dxdy\right| = o(N^3) $$
	for $i\in\{2,\,4,\,7\}$.
	
\end{lemma}

\begin{proof}[Proof of Lemma \ref{lem:oN3}]
	By a change of variables, we can write
	\begin{eqnarray*}
		&&\int_{[-N, N]^2}\int_{[-N, N]^2} g_i(\|x-y\|)dxdy = \\
		&& \qquad\int_{-2N}^{2N}\int_{-2N}^{2N} (2N-|x_1|)(2N-|x_2|)g_i\left(\left\| \begin{pmatrix}
			x_1 \\ 
			x_2
	\end{pmatrix}\right\|\right)dx_1 dx_2.
	\end{eqnarray*}
	This integral can be decomposed into the integral over the inscribed ball $B_{2N}(0)$ and the integral over the remaining domain. The integral over $B_{2N}(0)$ can be rewritten with spherical coordinates, and the bound then follows by the second statement of Lemma \ref{leading} and Remark \ref{1st-tech}. As for the remaining domain (the four ``corners'' of the rectangle), note that $g_2\left(\left\| \begin{pmatrix} x_1 \\   x_2 \end{pmatrix}\right\|\right)$, $g_4\left(\left\| \begin{pmatrix}  x_1 \\   x_2\end{pmatrix}\right\|\right)$ and $g_7\left(\left\| \begin{pmatrix}  x_1 \\  x_2 \end{pmatrix}\right\|\right)$ are radial functions that decay as $\frac{1}{\sqrt{N}}$ as $N$ tends to infinity. Moreover, their zero sets are (asymptotically) concentric circles with radii that are multiples of $2\pi$. Due to these properties, and also since the area of the remaining domain decays with distance from the origin, we can bound the integral over the remainder by two times the integral between two zero set circles closest to $\partial B_{2N}(0)$ of the absolute value of $g_2$, $g_4$ or $g_7$ times the factor $(2N-|x_1|)(2N-|x_2|)$. This integral is asymptotically bounded by the area of integration ($2\pi (2N) \cdot 2\pi$) multiplied by a bound over the integrand $(2N)^2\frac{1}{\sqrt{N}}$. This proves the statement also for the remaining domain.
\end{proof}

\begin{lemma}\label{lem:dominant}
	We have
	\begin{eqnarray*}
		&& \int_{-2N}^{2N}\int_{-2N}^{2N} (2N-|x_1|)(2N-|x_2|)g_2^2\left(\left\| \begin{pmatrix}
			x_1 \\ 
			x_2
		\end{pmatrix}\right\|\right)dx_1 dx_2\\
		&&\qquad \sim \int_{-2N}^{2N}\int_{-2N}^{2N} (2N-|x_1|)(2N-|x_2|)g_4^2\left(\left\| \begin{pmatrix}
			x_1 \\ 
			x_2
		\end{pmatrix}\right\|\right)dx_1 dx_2\\
		&&\qquad \sim \int_{-2N}^{2N}\int_{-2N}^{2N} (2N-|x_1|)(2N-|x_2|)g_7^2\left(\left\| \begin{pmatrix}
			x_1 \\ 
			x_2
		\end{pmatrix}\right\|\right)dx_1 dx_2\\&&\qquad =4   (2N)^3\,\frac{1}{\pi} \left(\frac{1-\sqrt{2}}{3} + \ln(1 + \sqrt{2}) \right)   +O(N^2\sqrt{N})
	\end{eqnarray*}
	
\end{lemma}

\begin{proof}[Proof of Lemma \ref{lem:dominant}]
	First note that we can infer from the recurrence relations between Bessel functions that $g_2(r)\sim g_7(r)\sim J_0(r)$ and $g_4(r)\sim -J_1(r)$ as $r\to\infty$. Due to the asymptotics \eqref{Bessel} we moreover obtain
	$$J_0^2(r)\sim  \frac{2}{\pi r}\cos^2(r-\pi/4)=  \frac{1}{\pi r} (\cos(2r-\pi/2)+1) $$
	as well as
	$$J_1^2(r)\sim  \frac{2}{\pi r}\cos^2(r-\pi/2-\pi/4)=  \frac{1}{\pi r} (\cos(2r+\pi/2)+1). $$
	Note that, since the integrals appearing in this lemma all tend to infinity with $N\to\infty$, replacing Bessel functions with their approximations up to an additive error of a smaller order is not going to change the asymptotics (nor the relevant constants).
	Therefore,
	\begin{eqnarray*}
		&& \int_{-2N}^{2N}\int_{-2N}^{2N} (2N-|x_1|)(2N-|x_2|)\left(g_2^2\left(\left\| \begin{pmatrix}
			x_1 \\ 
			x_2
		\end{pmatrix}\right\|\right)-g_4^2\left(\left\| \begin{pmatrix}
			x_1 \\ 
			x_2
		\end{pmatrix}\right\|\right)\right)dx_1 dx_2\\ 
		&&\qquad \sim \int_{-2N}^{2N}\int_{-2N}^{2N} (2N-|x_1|)(2N-|x_2|) \frac{1}{\pi \left\| \begin{pmatrix}
				x_1 \\ 
				x_2
			\end{pmatrix}\right\|}2\cos \left(2\left\| \begin{pmatrix}
			x_1 \\ 
			x_2
		\end{pmatrix}\right\| +\frac{\pi}{2}\right)dx_1 dx_2,
	\end{eqnarray*}
	and we can see with the same argument as in the previous lemma that this integral is of order $o(N^3)$.\\
	To show that all the integrals are indeed of order $N^3$, one can use the fact that the oscillating part (involving $\cos(2r+\pi/2)$ and considered in the computation above) is not contributing to the asymptotics. 
	Indeed from the asymptotics above we have 
	\[ g_2^2(r)= \frac{1+\sin (2r)}{\pi r}+O\left(\frac{1}{r^2}\right)\]
	as $r\to \infty$.
	
	This allows us to write explicitly, as $N\to\infty$,
\begin{eqnarray*}
    &&\int_{-2N}^{2N}\int_{-2N}^{2N} (2N-|x_1|)(2N-|x_2|)g_2^2\left(\left\| \begin{pmatrix}  x_1 \\   x_2\end{pmatrix}\right\|\right)dx_1 dx_2\\
    && \sim  \int_{-2N}^{2N}\int_{-2N}^{2N} (2N-|x_1|)(2N-|x_2|) \frac{1}{\pi \left\| \begin{pmatrix}
  x_1 \\ 
  x_2
\end{pmatrix}\right\|}dx_1 dx_2\\
&& = 4 \int_{0}^{2N}\int_{0}^{2N} (2N-x_1)(2N-x_2) \frac{1}{\pi \sqrt{x_1^2+x_2^2}}dx_1 dx_2\\
&& = 4 \cdot (2N)^2\int_{0}^{1}\int_{0}^{1} (2N-2N y_1)(2N-2N y_2) \frac{1}{\pi\cdot 2N\sqrt{y_1^2+y_2^2}}dy_1 dy_2\\
&&  = 4   (2N)^3\int_{0}^{1}\int_{0}^{1} (1- y_1)(1-y_2) \frac{1}{\pi \sqrt{y_1^2+y_2^2}}dy_1 dy_2 \\
&& =4   (2N)^3\,\frac{1}{\pi} \left(\frac{1-\sqrt{2}}{3} + \ln(1 + \sqrt{2})\right)    = 4   (2N)^3\, 
\,\frac{1}{\pi} \left(\frac{1-\sqrt{2}}{3} +\ln(1 + \sqrt{2})\right) \approx 4   (2N)^3\, 0.2366
\end{eqnarray*}

\end{proof}
 The following Lemma \ref{coefficients} provides the coefficients of the dominating chaos components (see proof of Proposition \ref{th:variance} for the precise definition); the proof exploits \cite[Proposition 5]{CM18} and we omit it for the sake of brevity. 

\begin{lemma}\label{coefficients}
	We have
	\begin{eqnarray*}
		&&c_2 = -\frac{1}{\sqrt{\pi}}\frac{1}{4}u\phi(-u),\\
		&& c_{5} = \frac{1}{\sqrt{\pi}}\left(\frac{1}{6}u\phi (-u)+\frac{1}{6}u^3\phi (-u)+\frac{1}{4}\Phi (-u)\right),\\
		&& c_{z}= \frac{1}{\sqrt{\pi}}\left(\frac{1}{12}u\phi (-u)+\frac{1}{12}u^3\phi (-u)\right),\\
		&& c_{z5}= \frac{1}{\sqrt{\pi}}\left(\frac{1}{3\sqrt{8}}u\phi (-u)+\frac{1}{3\sqrt{8}}u^3\phi (-u)+\frac{1}{\sqrt{8}}\Phi (-u)\right),\\
		&& c_{25}=c_{2z}=0.
	\end{eqnarray*}
\end{lemma}

\section*{Funding}
The author E. D.B.  has been supported by the French government, through the 3IA C\^{o}te d'Azur Investments in the Future project managed by the National Research Agency (ANR) with the reference number ANR-19-P3IA-0002. The author A.P. T. is a member of INdAM-GNAMPA. 

\bibliographystyle{plain}

\end{document}